\documentclass{amsart}

\usepackage{url}
\usepackage{amsfonts}
\usepackage{graphicx}
\usepackage{epstopdf}
\usepackage{algorithmic}
\usepackage{dsfont}
\usepackage{hyperref}

\usepackage{amsmath,amssymb,amsthm}
\usepackage{epsfig,psfrag,url}
\usepackage{graphicx,subfigure}
\usepackage[utf8]{inputenc}
\usepackage{subfigure,enumerate}
\usepackage[norelsize]{algorithm2e}

\def\XXint#1#2#3{{\setbox0=\hbox{$#1{#2#3}{\int}$}
     \vcenter{\hbox{$#2#3$}}\kern-.5\wd0}}

\DeclareMathOperator{\tr}{tr}

\newenvironment{mat}{\left[\begin{array}{ccccccccccccccc}}{\end{array}\right]}
\newcommand\bcm{\begin{mat}}
\newcommand\ecm{\end{mat}}

\newcommand{\bea}{\begin{eqnarray}}
\newcommand{\eea}{\end{eqnarray}}
\newcommand{\bean}{\begin{eqnarray*}}
\newcommand{\eean}{\end{eqnarray*}}
\newcommand{\ba}{\begin{array}}
\newcommand{\ea}{\end{array}}
\newcommand{\beqs}{\begin{equation*}\begin{split}}

\newtheorem{remark}{Remark}[section]

\newtheorem{condition}{Condition}[section]

\long\def\symbolfootnote[#1]#2{\begingroup%
\def\thefootnote{\fnsymbol{footnote}}\footnote[#1]{#2}\endgroup}

\newcommand{\D}{\mathrm{d}}

\newcommand{\ds}{\displaystyle}

\newtheorem{theorem}{Theorem}
\newtheorem{lemma}{Lemma}
\newtheorem{proposition}{Proposition}
\newtheorem{definition}{Definition}
\newtheorem{corollary}{Corollary}

\thanks{The authors would like to thank Folkmar Bornemann for the data to display $F_{2}^\mathrm{gap}(t)$. {This work was supported in part by grants NSF-DMS-1303018 (TT) and NSF-DMS-1300965 (PD).}}
\newcommand{\TheTitle}{Universality for eigenvalue algorithms on sample covariance matrices}

\title{{\TheTitle}}

\author{Percy Deift}
\address{Courant Institute of Mathematical Sciences,
New York University,
251 Mercer St.,
New York, NY 10012, USA}
\email{deift@cims.nyu.edu}
\author{Thomas Trogdon}
\address{Department of Mathematics,
University of California, Irvine,
Irvine, CA 92697-3875, USA}
\email{ttrogdon@math.uci.edu}

\usepackage{amsopn}
\DeclareMathOperator{\diag}{diag}

\renewcommand{\ds}{}


\begin{document}

\begin{abstract}
We prove a universal limit theorem for the halting time, or iteration count, of the power/inverse power methods and the QR eigenvalue algorithm. Specifically, we analyze the required number of iterations to compute extreme eigenvalues of random, positive-definite sample covariance matrices to within a prescribed tolerance.  The universality theorem provides a complexity estimate for the algorithms which, in this random setting, holds with high probability.  The method of proof relies on recent results on the statistics of the eigenvalues and eigenvectors of random sample covariance matrices (i.e., delocalization, rigidity and edge universality).
\end{abstract}

\maketitle


\section{Introduction}

In this paper, we prove a universal limit theorem for the fluctuations in the runtime (or halting time) of three classical eigenvalue algorithms applied to positive-definite random matrices.  The theorem is universal in the sense that the limiting distribution does not depend on the distribution of the entries of the matrix (within a class).

One can trace the search for universal behavior in eigenvalue algorithm runtimes to the largely-experimental work of Pfrang, Deift and Menon \cite{DiagonalRMT}.  The authors considered three algorithms (QR, matrix sign and Toda) and ran the algorithms to the time of first deflation, which we now describe in more detail.  Given an $N \times N$ matrix $H$, the algorithms produce isospectral iterates $X_n$, $X_0 = H$, $\mathrm{spec}\,X_n = \mathrm{spec}\, H$, and generically $X_n \to \diag(\lambda_1,\ldots,\lambda_N)$.  Necessarily, the $\lambda_i$'s are the eigenvalues of $H$. However, one does not typically run the algorithm until the norm of all of the off-diagonal entries is small. Rather, one considers the submatrices $X_n^{(k)}$ which consist of the entries of $X_n$ that are in the first $k$ rows and the last $N-k$ columns. The $k$-\emph{deflation times} are defined as
\begin{align*}
T^{(k)}(H) := \min\{n: \|X_n^{(k)}\|_{\mathrm{F}} \leq \epsilon\}, \quad \epsilon >0.
\end{align*}
Here $\|\cdot\|_{\mathrm{F}}$ denotes the Frobenius norm\footnote{The authors in \cite{DiagonalRMT} actually considered a scaled $\infty$-norm rather than the Frobenius norm.}.  Then the time of first deflation is given by $T_{\mathrm{def}}(H) := \min_{1 \leq k \leq N-1} T^{(k)}(H)$.  We define $\hat k = \hat k(H)$ to be the largest value of $k$ such that $T^{(k)}(H) = T_\mathrm{def}(H)$.  It follows that when $k = \hat k(H)$, the eigenvalues of the leading $k \times k$ and $(N- k)\times (N - k)$ submatrices approximate the eigenvalues of $H$ to $\mathcal O(\epsilon)$ .  The algorithm is then applied to the smaller submatrices, and so on.

A typical experiment from \cite{DiagonalRMT} goes as follows.  Let $Y_\mathrm{G}$ and $Y_\mathrm{B}$ be $N \times N$ matrices of iid standard normal and iid mean-zero, variance-one Bernoulli random variables, respectively.  Then define $H_\mathrm{G} = (Y_\mathrm{G} + Y_\mathrm{G}^T)/\sqrt{2N}$ and $H_\mathrm{B} = (Y_\mathrm{B} + Y_\mathrm{B}^T)/\sqrt{2N}$ which are real, symmetric random matrices (see \cite{Deift2014} for complex Hermitian matrices).  After sampling the integer-valued random variables $T_{\mathrm{def}}(H_\mathrm{G})$ and $T_{\mathrm{def}}(H_\mathrm{B})$, for $N$ large, say 10,000 times, we define the empirical fluctuations
\begin{align*}
\bar T_{\mathrm{def}}(H_\mathrm{G}) = \frac{ T_{\mathrm{def}}(H_\mathrm{G}) - \langle T_{\mathrm{def}}(H_\mathrm{G}) \rangle }{\sigma_\mathrm{G}},\\
\bar T_{\mathrm{def}}(H_\mathrm{B}) = \frac{ T_{\mathrm{def}}(H_\mathrm{B}) - \langle T_{\mathrm{def}}(H_\mathrm{B}) \rangle }{\sigma_\mathrm{B}}.
\end{align*}
where $\langle \cdot \rangle$ and $\sigma_{(\cdot)}$ represent the sample mean and sample standard deviation, respectively.  We plot the histograms for the empirical fluctuations in Figure~\ref{f:QRdeflation}.  The histograms overlap surprisingly well for the two different ensembles,  indicating that after centering and rescaling, the distribution of the time of first deflation is universal.

\begin{figure}[tbp]
  \centering
\subfigure[]{\includegraphics[width=.49\linewidth]{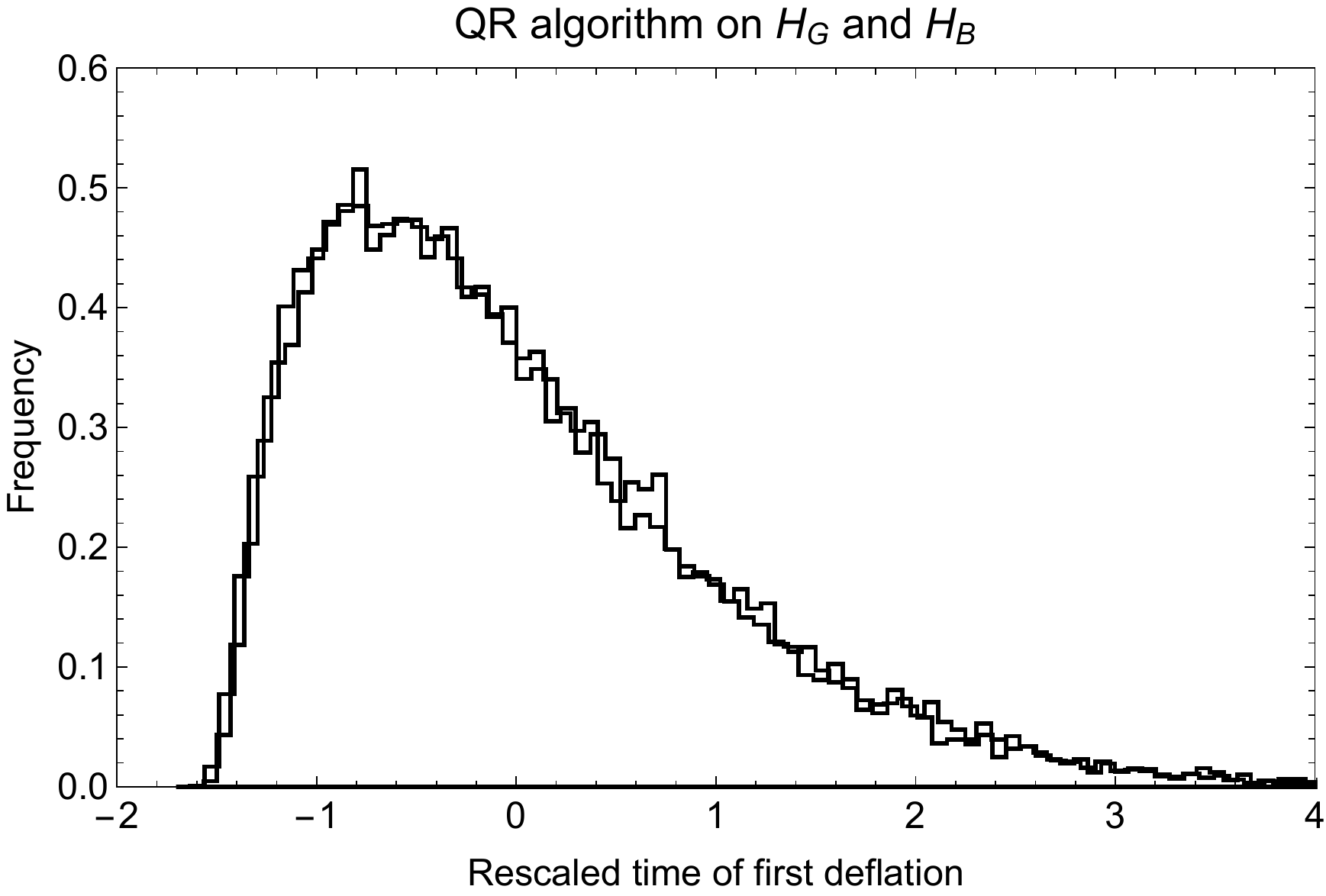}\label{f:QRdeflation}}
\subfigure[]{\includegraphics[width=.49\linewidth]{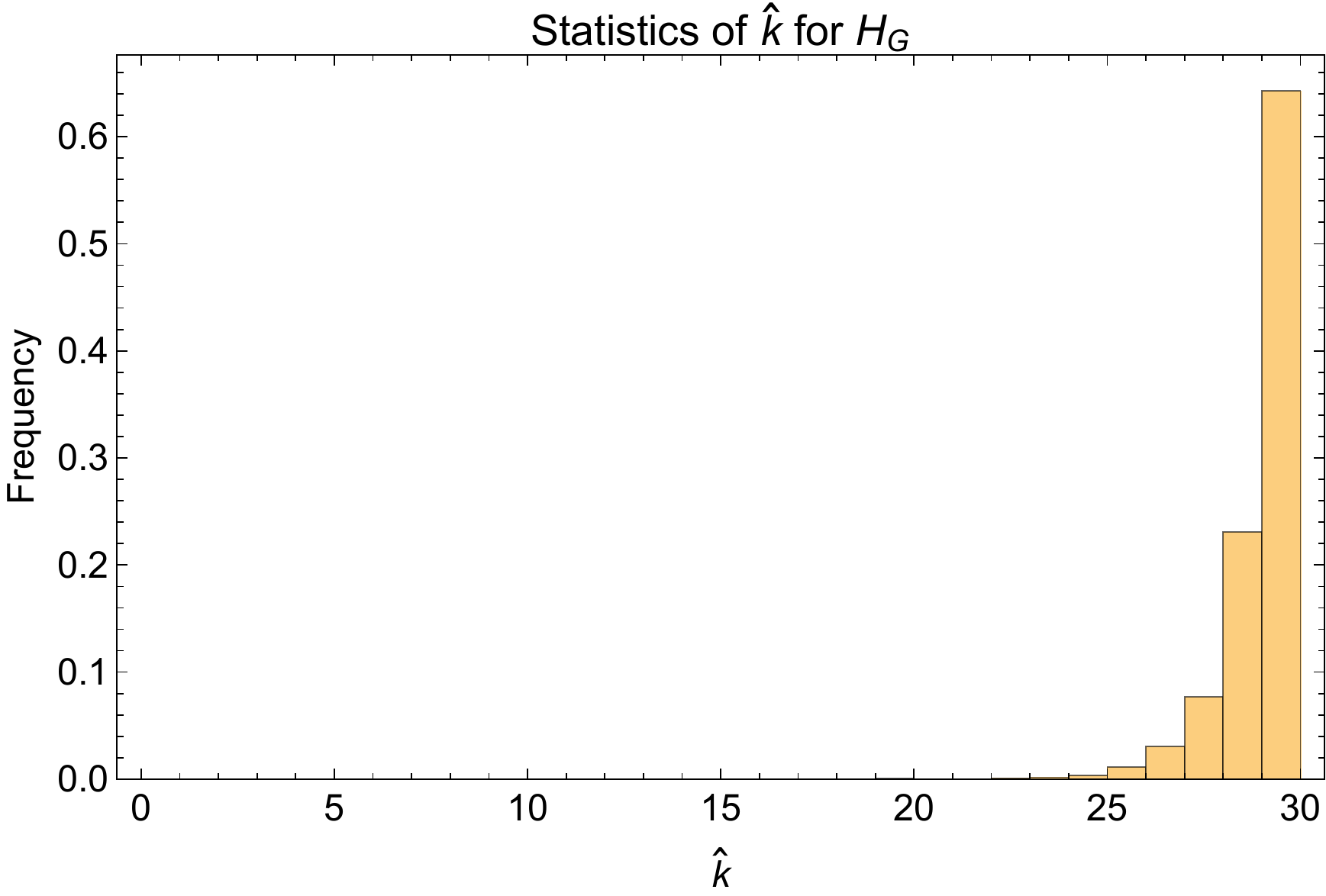}\label{f:hatk}}
\caption{(a) Sampled histograms for $\bar T_{\mathrm{def}}(H_\mathrm{G}) $ and $\bar T_{\mathrm{def}}(H_\mathrm{B})$. Dispite the differences in the underlying distributions for $H_{\mathrm G}$ and $H_{\mathrm B}$, the statistics for the shifted and scaled time of first deflation is the effectively the same. (b) The statistics of $\hat k$ when $N = 30$ for $H_{\mathrm G}$. It is clear that $\hat k = N-1$ occurs with the largest probability.}
\end{figure}

Proving theorems about the random variable $T_{\mathrm{def}}(H)$ is particularly difficult as one has to analyze the minimum of $N-1$ correlated random variables.  In \cite{Deift2016} the authors proved a (universal) limit theorem for the $1$-deflation time of the so-called Toda algorithm.  In this paper we prove an analog of that result for the $(N-1)$-deflation time for the QR (eigenvalue) algorithm acting on positive definite matrices.  This is an important first step in proving a limit theorem for  $T_{\mathrm{def}}(H)$  because it is the most likely that $\hat k = N-1$, see Figure~\ref{f:hatk}.  We also include similar results for the power (P) and inverse power (IP) methods as the analysis is similar. For these two methods, we incorporate random starting vectors. The analysis and results of the current work are quite similar to that in \cite{Deift2016}, where we prove universality for the Toda eigenvalue algorithm, showing its wide applicability.

\subsection{Relation to previous work and complexity theory}

The statistical analysis of algorithms has been performed in many settings, usually with an eye towards complexity theory.  In relation to Gaussian elimination, the seminal work is the analysis of Goldstine and von Neumann \cite{Goldstine1951} on the condition number of random matrices.   This is closely related to the later work of Edelman \cite{Edelman1988}, also on condition numbers.  The expected number of pivot steps in the simplex algorithm was analyzed by Smale \cite{Smale1985} and Borgwardt \cite{Borgwardt1987}.  The methodology of smoothed analysis was introduced in \cite{Spielman2001} and applied in a variety of settings \cite{Manthey2007,Menon2016,Sankar2006}.

The closest work, within the realm of complexity theory, to the current work is that of Kostlan \cite{Kostlan1988}.  Kostlan showed that for the power method on $H_{G}$ the expected halting time to compute an \underline{eigenvector} is infinite.  Kostlan showed that when one conditions on all of the eigenvalues being positive, the upper bound on the halting time is $\mathcal O(N^2 \log N)$.  Instead of conditioning, and eigenvector computation, we turn to sample covariance matrices which (with high probability) have positive eigenvalues and use the power methods to compute the extreme eigenvalues.  With this we are able to determine the precise limiting distribution of the halting time, which contains far more information than simply an upper bound.  To our knowledge, this is the first time this has been done for a classical numerical method.

For $\alpha$ in the scaling region given by Condition~\ref{cond:scaling}, the halting times given in Theorem~\ref{t:main} scale like $(\alpha - 2/3)N^{2/3}\log N$ in order to obtain an accuracy of $N^{-\alpha/2}$. This is a key conclusion of our results which gives an estimate on the complexity of the QR algorithm, and also the power and inverse power methods.

  Through many detailed computations, universality in numerical computation has been observed in many numerical algorithms beyond the QR algorithm and the power and inverse power methods (see  \cite{Deift2016,Deift2014,Deift2015,DiagonalRMT,Sagun2015} ): the conjugate gradient algorithm, the matrix sign eigenvalue algorithm, the Toda eigenvalue algorithm, the Jacobi eigenvalue algorithm, the GMRES algorithm, a genetic algorithm and the gradient and stochastic gradient descent algorithms. This work  presents further examples, in addition to \cite{Deift2016}, where one can prove this type of universality.  This advances the contention of the authors that universality is a bona fide and basic phenomenon in numerical computation.

\subsection{Open questions}

The main open question related to this work is the asymptotics of the time of first deflation $T_\mathrm{def}$.  A related and unknown detail is the tail behavior of the limiting distribution. As discussed in detail in \cite{Deift2016}, the limiting distribution in Theorem~\ref{t:main} in \cite{Deift2016} for the halting time has one finite moment for real matrices and two finite moments for complex matrices.  If one constructed an algorithm with a sub-Gaussian limiting distribution, it may be preferable. We believe this is the case for $T_\mathrm{def}$.  We also believe that its distribution is related to the largest gap in the spectrum of the stochastic Airy operator \cite{Ramirez2011}.  Furthermore, can one extend our results for the QR algorithm to indefinite ensembles?

We only consider random matrices with entries that are exponentially localized, see \ref{e:subexp}.  It is not known if this condition can be relaxed but it is itself an important open question.  Finally, additional halting criteria can be employed.  One could look for the time to compute eigenvectors with the power method, or to compute the entire spectrum with the QR algorithm.

\section{Main results}

In this paper we discuss computing the smallest and largest eigenvalues of random positive definite matrices to an accuracy $\epsilon$.  We have a basic condition that we enforce on $\epsilon$ which requires that $\epsilon$ is appropriately small.
\begin{condition}\label{cond:scaling}
\begin{align*}
\frac \alpha 2 : = \log \epsilon^{-1}/\log N \geq 5/3 + \sigma/2,
\end{align*}
for $0 < \sigma < 1/3$ fixed.
\end{condition}

For $j = 0,1,2,\dots$, we let $X_j$ be the iterates of the QR algorithm (QR, defined in Section~\ref{sec:QR}) and $\lambda_{\mathrm{P},j}$ and $\lambda_{\mathrm{IP},j}$ be the iterates of the power and inverse power methods, respectively (P and IP, respectively, defined in Section~\ref{sec:Power}).  We specify (discrete) halting times for these algorithms applied to a matrix $H$ with starting vector $v$ as follows:
\begin{align*}
\tau_{\mathrm{QR}, \epsilon}(H) &:= \min \left\{ j: \sum_{n=1}^{N-1} |[X_j]_{Nn}|^2 \leq \epsilon^2 \right\},\\
\tau_{\mathrm{P}, \epsilon}(H,v) &:= \min \{ j: |\lambda_{\mathrm{P},j}- \lambda_{\mathrm{P},j+1}| \leq \epsilon^2\},\\
\tau_{\mathrm{IP}, \epsilon}(H,v) &:= \min \{ j: |\lambda_{\mathrm{IP},j}^{-1}- \lambda^{-1}_{\mathrm{IP},j+1}| \leq \epsilon^2\}.
\end{align*}
Note that for the QR algorithm the $(N,N)$ entry of $X_j$, $[X_j]_{NN}$, is an approximation of the smallest eigenvalue $\lambda_1$ as is $\lambda_{\mathrm{IP},j}$. On the other hand, $\lambda_{\mathrm{P},j}$ is an approximation of the largest eigenvalue $\lambda_N$. Our main results are summarized in the following Theorem and Propsosition.  See Definition~\ref{def:WE} for the definition of sample covariance matrices and Definition~\ref{def:Fbeta} for the distribution function $F_\beta^{\mathrm{gap}}(t)$.  The constants $\lambda_\pm$ and $d$ are given in \eqref{e:MP}.

\begin{theorem}[Universality]\label{t:main}
Let $H$ be a real ($\beta = 1$) or complex ($\beta = 2$) $N \times N$ sample covariance matrix and let $v$ be a (random) unit vector independent of $H$.  Assuming $\epsilon$ satisfies Condition~\ref{cond:scaling}, for $t \in \mathbb R$
\begin{align*}
F_\beta^{\mathrm{gap}}(t) &= \lim_{N \to \infty} \mathbb P \left( \frac{\displaystyle \tau_{\mathrm{QR},\epsilon}(H)}{ \displaystyle  2^{-7/6} \lambda_-^{1/3} d^{-1/2} N^{2/3}  (\log \epsilon^{-1} - 2/3 \log N)}  \leq t\right)\\
&= \lim_{N \to \infty} \mathbb P \left( \frac{\displaystyle \tau_{\mathrm{IP},\epsilon}(H,v)}{ \displaystyle  2^{-7/6} \lambda_-^{1/3} d^{-1/2} N^{2/3}  (\log \epsilon^{-1} - 2/3 \log N)}  \leq t\right)\\
&= \lim_{N \to \infty} \mathbb P \left( \frac{\displaystyle \tau_{\mathrm{P},\epsilon}(H,v)}{ \displaystyle  2^{-7/6} \lambda_+^{1/3} d^{-1/2} N^{2/3}  (\log \epsilon^{-1} - 2/3 \log N)}  \leq t\right).
\end{align*}
\end{theorem}
This theorem is a direct consequence of Theorems~\ref{t:QR-tech} and \ref{t:Power-tech}, after noting that, for example, $|\tau_{\mathrm{QR}, \epsilon} - T_{\mathrm{QR}, \epsilon}| \leq 1$ where $T_{\mathrm{QR}, \epsilon}$ appears in Theorem~\ref{t:QR-tech}. This is a universality theorem in the sense that it states that for $N$ is sufficiently large the distribution of the halting time is independent of the distribution on $H$.

The following Proposition shows that we obtain an accuracy of $\epsilon$ but not of $\epsilon^2$, \emph{i.e.} that our halting criteria are sufficient but not too restrictive.  It is a restatement of Propositions~\ref{p:QRtrue} and \ref{p:Powertrue}.
\begin{proposition}\label{p:main}
Assuming $\epsilon$ satisfies Condition~\ref{cond:scaling}, for any real or complex sample covariance matrix
\begin{align*}
\epsilon^{-1}| [X_{\tau_{\mathrm{QR}, \epsilon}}]_{NN} - \lambda_1|, ~~ \epsilon^{-1}| \lambda_{\mathrm{IP},\tau_{\mathrm{IP}, \epsilon}} - \lambda_1|, ~~ \text{ and } ~~ \epsilon^{-1}| \lambda_{\mathrm{P},\tau_{\mathrm{P}, \epsilon}} - \lambda_N|
\end{align*}
converge to zero in probability, while
\begin{align*}
\epsilon^{-2}| [X_{\tau_{\mathrm{QR}, \epsilon}}]_{NN} - \lambda_1|, ~~ \epsilon^{-2}| \lambda_{\mathrm{IP},\tau_{\mathrm{IP}, \epsilon}} - \lambda_1|, ~~ \text{ and } ~~ \epsilon^{-2}| \lambda_{\mathrm{P},\tau_{\mathrm{P}, \epsilon}} - \lambda_N|
\end{align*}
converge to $\infty$ in probability.
\end{proposition}
A numerical demonstration of Theorem~\ref{t:main} is given in Section~\ref{sec:numerics}.

The outline of the paper is as follows.  In Section~\ref{sec:RMT} we discuss the fundamental results of random matrix theory that are required to prove our results.  In Section~\ref{sec:numerics} we give a numerical demonstration of Theorem~\ref{t:main}.  Next, in Section~\ref{sec:alg}, we discuss the fundamentals of the power methods and the QR algorithm before we apply the random matrix estimates in Section~\ref{sec:proofs} to prove our results. In Appendix~\ref{sec:error} we analyze the true error of the methods with our chosen halting criteria to see that these criteria are indeed appropriate to the task.  Finally, in Appendix~\ref{sec:proj} we discuss the asymptotic normality of eigenvector projections of random vectors.  This allows us to show that Theorem~\ref{t:main} indeed holds for random starting vectors in the power and inverse power methods.

\section{Results from random matrix theory}\label{sec:RMT}

We now introduce the ideas and results from random matrix theory that are needed to prove our main theorems.  Let $V$ be an $M \times N$ real or complex matrix with $M \geq N$. We consider the ordered eigenvalues $\lambda_j(H) = \lambda_j$, $j = 1,2,\ldots,N$ of $H = V^*V/M$, $\lambda_1 \leq \lambda_2 \leq \cdots \leq\lambda_N$.  Let $\beta_1, \beta_2, \ldots, \beta_{N}$ denote the absolute value of the last components of the associated normalized eigenvectors.  We only consider sample covariance matrices from independent samples.

\begin{definition}[Sample covariance matrix (SCM)]\label{def:WE}
A sample covariance matrix (ensemble) is a real symmetric ($\beta = 1$) or complex Hermitian ($\beta =2$) matrix $H = V^*V/M$, $V = (V_{ij})_{1\leq i \leq M, 1 \leq j \leq N}$ such that $V_{ij}$ are independent random variables for $1 \leq i \leq M$, $1 \leq j \leq N$ given by a probability measure $\nu_{ij}$ with
\begin{align*}
\mathbb E V_{ij} = 0, \quad  \mathbb E |V_{ij}|^2 = 1,
\end{align*}
Next, assume there is a fixed constant $\nu$ (independent of $N,i,j$) such that
\begin{align}\label{e:subexp}
\mathbb P(|V_{ij}| > x) \leq \nu^{-1} \exp(-x^\nu),\quad x > 1.
\end{align}
For $\beta = 2$ (when $V_{ij}$ is complex-valued) the condition
\begin{align*}
\mathbb E V_{ij}^2 = 0,
\end{align*}
must also be satisfied.
\end{definition}

We assume all SCMs have $M \geq N$.  Define the averaged empirical spectral measure
\begin{align*}
\mu_N(z) = \mathbb E \frac{1}{N} \sum_{i=1}^N \delta(\lambda_i-z),
\end{align*}
where the expectation is taken with respect to the given ensemble.  For technical reasons we let $M = M(N)$ and $d_N := N/M$ satisfy $\lim_{N \to \infty} d_N =:d \in (0,1)$.  More specifically, we consider $M = \lfloor N/d \rfloor$.

\begin{remark}
  The case where $\lim_{N \to \infty} d_N =1$ is of considerable interest:  If $M = N + R$ then it is known that the limiting distribution of the smallest eigenvalue is given in terms of the so-called Bessel kernel \cite{BenArous2005,Forrester1993} when $X_{ij}$ has Gaussian divisible entries. If $R \to \infty$, $R \leq  CN^{1/2}$ and $X_{ij}$ are standard complex normal random variables then it is known that the smallest eigevalue has Tracy--Widom fluctuations \cite{Deift2015}.  It is noted in \cite[Section 1.4]{Pillai2014} that establishing all estimates we use below in the $\lim_{N \to \infty} d_N =1$ case is a difficult problem.  In light of the current work, this is a particularly interesting problem as it would give different scalings for the halting times.
\end{remark}

Define the Marchenko--Pastur law
\begin{align}\label{e:MP}
  \rho_d(x) := \frac{1}{2 \pi d} \sqrt{\frac{[(\lambda_+ - x)(x-\lambda_-)]_+}{x^2}}, \quad
  \lambda_\pm &= (1 \pm \sqrt{d})^2,
\end{align}
and $[\cdot]_+$ denotes the positive part.  For SCMs, $\mu_N$ converges to $\rho_d(x) \D x$ weakly and $\rho_d(x) \D x$ is called the \emph{equilibrium measure} for the ensemble (see, for example, \cite{Marcenko1967,Pillai2014,Silverstein1986,Wachter1978,Yin1986}).

\begin{definition}\label{d:quant}
Define $\gamma_n$ to be the smallest value of $t$ such that
\begin{align*}
\frac{n}{N} = \int_{-\infty}^t \rho_d(x) \D x, \quad n = 1,2,\ldots,N.
\end{align*}
\end{definition}
Thus $\{\gamma_n\}$ represent the quantiles of the equilibrium measure.  We now describe conditions on the matrices that simplify the analysis of the algorithms QR, P and IP.

\begin{condition} \label{cond:uppergap} For $0 < p < \sigma/4$,
\begin{itemize}
  \item $\ds \frac{\lambda_{N-2}}{\lambda_{N-1}} < \left(  \frac{\lambda_{N-1}}{\lambda_N} \right)^p$.
\end{itemize}
Let $\mathcal U_{N,p}$ denote the set of matrices that satisfy this condition.
\end{condition}

\begin{condition} \label{cond:lowergap} For $0 < p < \sigma/4$,
\begin{itemize}
 \item $\ds \frac{\lambda_2}{\lambda_3} < \left(  \frac{\lambda_1}{\lambda_2} \right)^p$.
\end{itemize}
Let $\mathcal L_{N,p}$ denote the set of matrices that satisfy this condition.
\end{condition}

Given an SCM, let $v$ be a random (or deterministic) unit vector independent of the SCM.  Define $\beta_n = |\langle v, u_n \rangle|$, $n = 1,2,\ldots,N$ where $u_n$ is the $n$th eigenvector of the SCM.
\begin{condition}\label{cond:rigidity}
For any fixed $0<  s < \sigma/40$,
\begin{enumerate}
\item $\beta_n \leq N^{-1/2+s/2}$ for all $n$
\item $N^{-1/2-s/2} \leq \beta_n$ for $n = 1,2,N-1,N$,
\item $N^{-2/3-s/2} \leq \lambda_N - \lambda_{n-1} \leq N^{-2/3+s/2}$, for $n = N, N-1$,
\item $N^{-2/3-s/2} \leq \lambda_n - \lambda_{1} \leq N^{-2/3+s/2}$, for $n = 2, 3$, and
\item $|\lambda_n - \gamma_n| \leq N^{-2/3+s/2}( \min\{n,N-n+1\})^{-1/3}$ for all $n$.
\end{enumerate}
Let $\mathcal R_{N,s}$ denote the set of matrices that satisfy these conditions.
\end{condition}

\begin{remark}\label{r:quantiles}
Clearly the quantiles $\{\gamma_n\}$ lie in the interval $(\lambda_-, \lambda_+)$. Property (5) above implies, in particular, that for $N$ sufficiently large,
the eigenvalues $\{\lambda_n\}$ of matrices in $\mathcal R_{N,s}$ lie  in the interval $(\lambda_- - \eta, \lambda_+ + \eta)$ for any given $\eta>0$.
\end{remark}

The analysis of the eigenvalues of sample covariance matrices has a long history, beginning with the work of Mar\u{c}enko and Pastur \cite{Marcenko1967}.  The seminal work of Geman \cite{Geman1980} showed that for $M,N \to \infty$, $N/M \to y \in (0,\infty)$, the largest eigenvalue of an SCM converges a.s. to $\lambda_+$.  Silverstein \cite{Silverstein1985} established that for $M,N \to \infty$, $N/M \to y \in (0,1)$ the smallest eigenvalue converges  a.s. to $\lambda_-$ when $V_{ij}$ are iid standard normal random variables.  See \cite{Forrester1993,Johansson2000,Johnstone2001} for the first results on the fluctuations of the largest and smallest eigenvalues when $V_{ij}$ are iid (real or complex) standard normal distributions.  Universality for the eigenvalues of $\frac{1}{N} V^*V$ at the edges and in the bulk, was first proved by Ben Arous an Pech\'e \cite{BenArous2005} for Gaussian divisible ensembles, in the limit $N,M \to \infty$, $M = N + \nu$, $\nu$ fixed.  We reference \cite{Pillai2014} and \cite{Bloemendal2016} for the most comprehensive results.   Note that we require \eqref{e:subexp} which is stronger than the assumptions in \cite{Geman1980,Yin1986} which only require moment conditions.  Various limits of the eigenvectors have also been considered, see \cite{Bai2007,Silverstein1986}.  But we reference \cite{Bloemendal2016} for the full generality we need to prove our theorems.

\begin{theorem}\label{t:gap-limit}
  For SCMs
\begin{align*}
  N^{2/3}\lambda_+^{-2/3}d^{1/2}(\lambda_+ - \lambda_N, \lambda_+ - \lambda_{N-1}, \lambda_+-\lambda_{N-2})
  \end{align*}
   and
   \begin{align*}
      \quad N^{2/3}\lambda_-^{-2/3}d^{1/2}(\lambda_1 - \lambda_-, \lambda_2 - \lambda_{-}, \lambda_3-\lambda_{-})
\end{align*}
separately converge jointly in distribution to random variables $(\Lambda_{1,\beta},\Lambda_{2,\beta},\Lambda_{3,\beta})$ which are the smallest three eigenvalues of the so-called stochastic Airy operator. Furthermore, $(\Lambda_{1,\beta},\Lambda_{2,\beta},\Lambda_{3,\beta})$ are distinct with probability one.
\end{theorem}
\begin{proof}
The first statement follows from \cite[Theorem~8.3]{Bloemendal2016}.  The second statement follows from \cite[Theorem 1.1 \& Corollary 1.2]{Pillai2014}.  The fact that the eigenvalues of the stochastic Airy operator are distinct is shown in \cite[Theorem 1.1]{Ramirez2011}.
\end{proof}

\begin{definition}\label{def:Fbeta}
The distribution function $F^\mathrm{gap}_\beta(t)$, supported on $t \geq 0$ for $\beta =1,2$ is given by
\begin{align*}
  F^\mathrm{gap}_\beta(t) = \mathbb P\left( \frac{1}{\Lambda_{2,\beta}-\Lambda_{1,\beta}} \leq t \right) &= \lim_{N \to \infty} \mathbb P\left( \frac{1}{2^{-7/6}N^{2/3}\lambda_+^{-2/3}d^{-1/2}(\lambda_N-\lambda_{N-1})} \leq t \right)\\
   & =  \lim_{N \to \infty} \mathbb P\left( \frac{1}{2^{-7/6}N^{2/3}\lambda_-^{-2/3}d^{-1/2}(\lambda_2-\lambda_{1})} \leq t \right).
\end{align*}
\end{definition}

The remaining theorems in this section are compiled from results that have been obtained recently in the literature.  We use a simple lemma (see, for example, \cite[Lemma 3.2]{Deift2016}):
\begin{lemma}\label{l:high}
  If $X_N \to X$ in distribution\footnote{For convergence in distribution, we require that the limiting random variable $X$ satisfies $\mathbb P(|X|< \infty) = 1$.} as $N \to \infty$ then for any $R > 0$
  \begin{align*}
    \mathbb P(|X_N/a_N| < R) = 1 + o(1)
  \end{align*}
  as $N \to \infty$ provided that $a_N \to \infty$.
\end{lemma}
\begin{theorem}\label{t:generic}
For SCMs, Condition~\ref{cond:rigidity} holds with high probability as $N \to \infty$, that is, for any $s > 0$
\begin{align*}
\mathbb P(\mathcal R_{N,s}) = 1 + o(1),
\end{align*}
as $N \to \infty$.
\end{theorem}
\begin{proof}
It suffices to show that each of the sub-conditions 1-5 in Condition~\ref{cond:rigidity} hold with high probability.  Conditions~\ref{cond:rigidity}.1-2 hold with high probability directly by Proposition~\ref{p:componentbound}. Conditions~\ref{cond:rigidity}.3-4 hold with high probability by the joint convergence of the top (bottom) three eigenvalues in Theorem~\ref{t:gap-limit} and Lemma~\ref{l:high}.  Finally, Condition~\ref{cond:rigidity}.5 holds with high probability as a direct consequence of \cite[Theorem 3.3]{Pillai2014}.
\end{proof}

\begin{theorem}\label{t:p}
For SCMs,
\begin{align*}
\lim_{p \downarrow 0}\limsup_{N \to \infty} \mathbb P(\mathcal U_{N,p}^c)  = \lim_{p \downarrow 0}\limsup_{N \to \infty} \mathbb P(\mathcal L_{N,p}^c) = 0.
\end{align*}
\end{theorem}
\begin{proof}
It follows from Theorem~\ref{t:gap-limit} that
\begin{align*}
\lim_{N\to \infty} \mathbb P(\lambda_{3}-\lambda_2 < p(\lambda_2-\lambda_{1})) = \mathbb P(\Lambda_{3,\beta} -\Lambda_{2,\beta} < p(\Lambda_{2,\beta} -\Lambda_{1,\beta})).
\end{align*}
Then
\begin{align*}
\lim_{p \downarrow 0} \mathbb P(\Lambda_{3,\beta} -\Lambda_{2,\beta} < p(\Lambda_{2,\beta} -\Lambda_{1,\beta})) &= \mathbb P \left( \bigcap_{p > 0}\left\{\Lambda_{3,\beta} -\Lambda_{2,\beta} < p(\Lambda_{2,\beta} -\Lambda_{1,\beta}) \right\} \right) \\
&= \mathbb P(\Lambda_{3,\beta} = \Lambda_{2,\beta}).
\end{align*}
But from \cite[Theorem 1.1]{Ramirez2011} $\mathbb P(\Lambda_{3,\beta} = \Lambda_{2,\beta}) = 0$.   And so, it suffices to show that
\begin{align*}
\lim_{N\to \infty} \mathbb P(\lambda_{3}-\lambda_2 < p(\lambda_2-\lambda_{1}))  =  \lim_{N \to \infty}\mathbb P \left(  \frac{\lambda_2}{\lambda_3} < \left(  \frac{\lambda_1}{\lambda_2}  \right)^p \right).
\end{align*}
This will, in turn follow, if we show that
\begin{align*}
\Gamma_N := \lambda_3-\lambda_2 - p(\lambda_2-\lambda_{1}) + \lambda_-\left[   \frac{\lambda_2}{\lambda_3} - \left(  \frac{\lambda_1}{\lambda_2}  \right)^p \right]
\end{align*}
converges to zero in probability for $p$ fixed.  We set $\lambda_j = \lambda_- + N^{-2/3}\xi_j$ where $(\xi_1,\xi_2,\xi_3)$ converges jointly in distribution by Theorem~\ref{t:gap-limit}.  Let $B_R$ be the event $\| (\xi_1,\xi_2,\xi_3)\| \leq R$ and for $\delta > 0$ consider
\begin{align*}
\mathbb P ( |\Gamma_N| \geq \delta ) = \mathbb P(|\Gamma_N| \geq \delta, B_R) + \mathbb P(|\Gamma_N| \geq \delta, B_R^c).
\end{align*}
Given $B_R$, we perform a formal expansion
\begin{align*}
\frac{\lambda_2}{\lambda_3} - \left(  \frac{\lambda_1}{\lambda_2}  \right)^p = \lambda_-^{-1} N^{-2/3}(\xi_2-\xi_3) - p \lambda_-^{-1}N^{-2/3}(\xi_1- \xi_3) + \mathcal O(N^{-4/3}).
\end{align*}
Therefore, given $B_R$, $\Gamma_N$ tends to zero uniformly and we find
\begin{align*}
\limsup_{N \to \infty} \mathbb P (|\Gamma_N| \geq \delta) \leq \limsup_{N \to \infty} \mathbb P(B_R^c).
\end{align*}
Because of joint convergence (in distribution) of $(\xi_1,\xi_2,\xi_3)$, the right-hand side tends to zero as $R \to \infty$.  This establishes the result for $\mathcal L_{N,p}$.  Similar considerations yield the result for $\mathcal U_{N,p}$.
\end{proof}

\section{A numerical demonstration}\label{sec:numerics}

We include some numerical simulations that serve to demonstrate Theorem~\ref{t:main}. We include ideas that were discussed in detail in \cite{Deift2016}.  En route to proving Theorem~\ref{t:main} we perform the following approximation step for A = QR, IP or P
\begin{align*}
\tau_{\mathrm{A},\epsilon} = \underbrace{\tau_{\mathrm{A},\epsilon} - T_{\mathrm{A},\epsilon}}_{:=D_1}  + \underbrace{T_{\mathrm{A},\epsilon} - T^*_{\mathrm{A},\epsilon}}_{:=D_2} + T^*_{\mathrm{A},\epsilon}.
\end{align*}
where $T^*_{\mathrm{A},\epsilon}$ is given in \eqref{e:t-starQR} and \eqref{e:t-starIP} below.  The difference $D_1$ is always less than unity and the difference $D_2$ is $\mathcal O(N^{2/3})$ (see Proposition~\ref{p:conv}, for example). Then $T^*_{\mathrm{A},\epsilon}$ converges in distribution, after rescaling, to $F_\beta^{\mathrm{gap}}$ but it is clear from the proof of Theorem~\ref{t:Power-tech} that the rate of covergence is logarithmic, at best.  To improve the rate we note that
\begin{align}\label{e:main-correction}
F_\beta^{\mathrm{gap}}(t) &= \lim_{N \to \infty} \mathbb P \left( \frac{\displaystyle \tau_{\mathrm{A},\epsilon}(H)}{ \displaystyle  2^{-7/6} \lambda_\pm^{1/3} d^{-1/2} N^{2/3}  (\log \epsilon^{-1} - 2/3 \log N + \zeta_\mathrm{A})}  \leq t\right),
\end{align}
for \emph{any} constant $\zeta_A$.  Here $\lambda_+$ is taken if A = P and $\lambda_-$ is taken if A = QR, IP.  We choose $\zeta_{\mathrm{QR}}$ (cf. with $\zeta$ chosen in \cite{Deift2016}), using \eqref{e:t-starQR}, by
\begin{align*}
\zeta_{\mathrm{QR}} = \mathbb E[\log N^{2/3} (\lambda_2 - \lambda_1)].
\end{align*}
After examining \eqref{e:t-starIP}, we choose
\begin{align*}
\zeta_{\mathrm{IP}} = \mathbb E[\log N^{2/3} (\lambda_2 - \lambda_1)] - 3/2 \log \lambda_- + 1/2 \log 2.
\end{align*}
Then changing $\lambda_2 \to \lambda_{N-1}^{-1}$ and $\lambda_1 \to \lambda_N^{-1}$  in \eqref{e:t-starIP} we choose
\begin{align*}
\zeta_{\mathrm{P}} = \mathbb E[\log N^{2/3} (\lambda_N - \lambda_{N-1})] - 1/2 \log \lambda_+ + 1/2 \log 2.
\end{align*}
Despite the fact that these $\zeta_{\mathrm{A}}$'s are not constant, from Theorem~\ref{t:gap-limit} one should expect they have well-defined limits as $N \to \infty$.  These effective constants can be easily approximated by sampling the associated matrix distributions.

In Figure~\ref{f:QRdemon} we demonstrate \eqref{e:main-correction} and hence Theorem~\ref{t:main} for the QR algorithm.  Figures~\ref{f:IPdemon} and \ref{f:Pdemon} demonstrate the analogous results for the inverse power method and power method, respectively. The ensembles we use are the following:
\begin{itemize}
\item[LOE]: $V$ (in Definition~\ref{def:WE} below) has iid standard real Gaussian entries,
\item[LUE]: $V$ has iid standard complex Gaussian entries,
\item[BE]: $V$ has iid mean-zero, variance-one Bernoulli entries ($\pm 1$ with equal probability),
\item[CBE]: $V$  has iid mean-zero, variance-one complex Bernoulli entries ( $\{a,-a, \bar a, -\bar a\}$, $a = (1+i)/2$, with equal probability)
\end{itemize}
The density $\frac{d}{dt} F_1^{\mathrm{gap}}(t)$ was computed by the authors in \cite{Deift2016}.  We sample the matrix distributions for $N$ large and use appropriate interpolation.  The density $\frac{d}{dt} F_2^{\mathrm{gap}}(t)$ was computed in \cite{Witte2013} (and rescaled in \cite{Deift2016}) and the data to reproduce it here was provided by the authors of that work.

\begin{figure}[tbp]
\subfigure[]{\includegraphics[width=.49\linewidth]{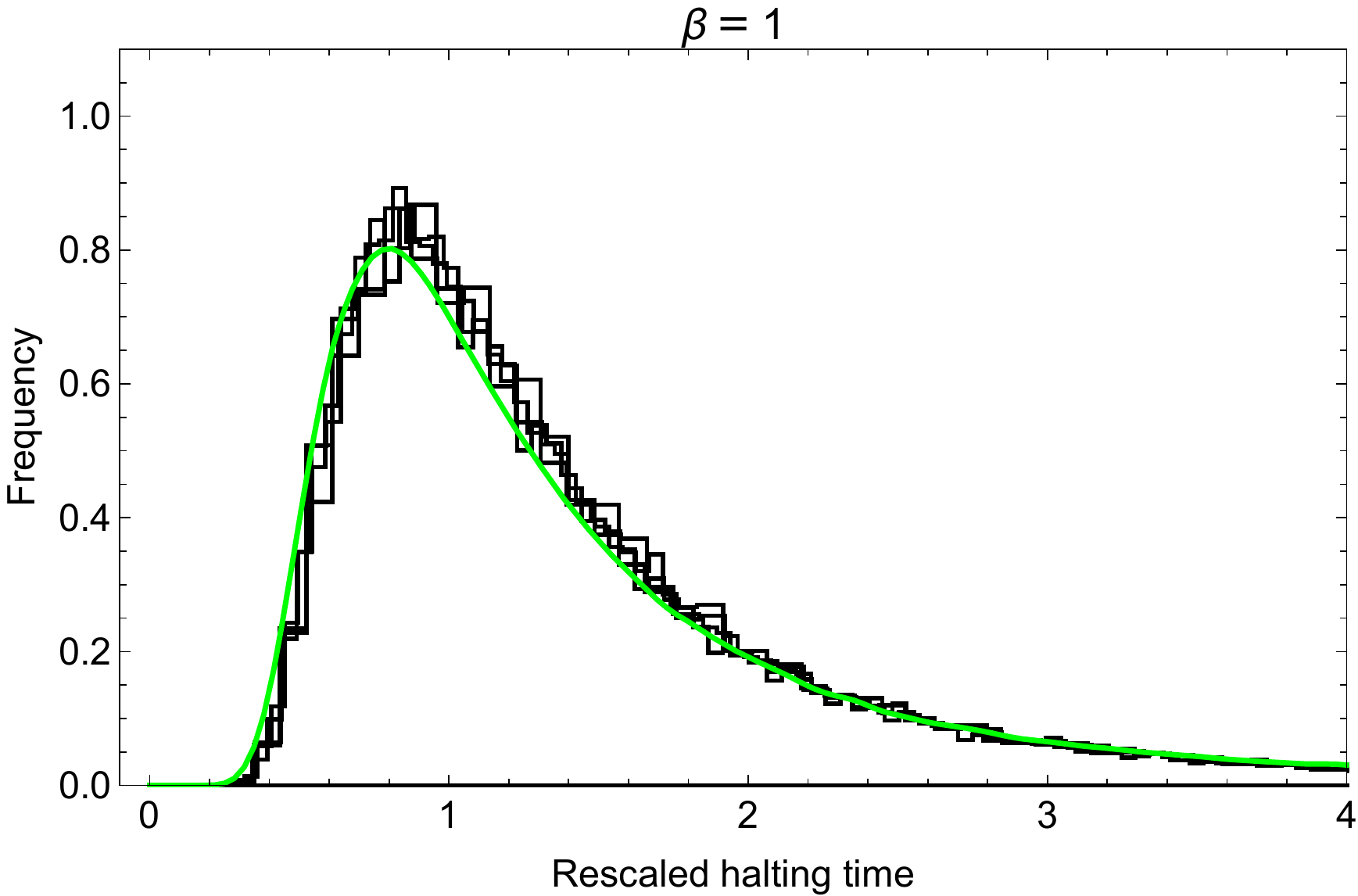}}
\subfigure[]{\includegraphics[width=.49\linewidth]{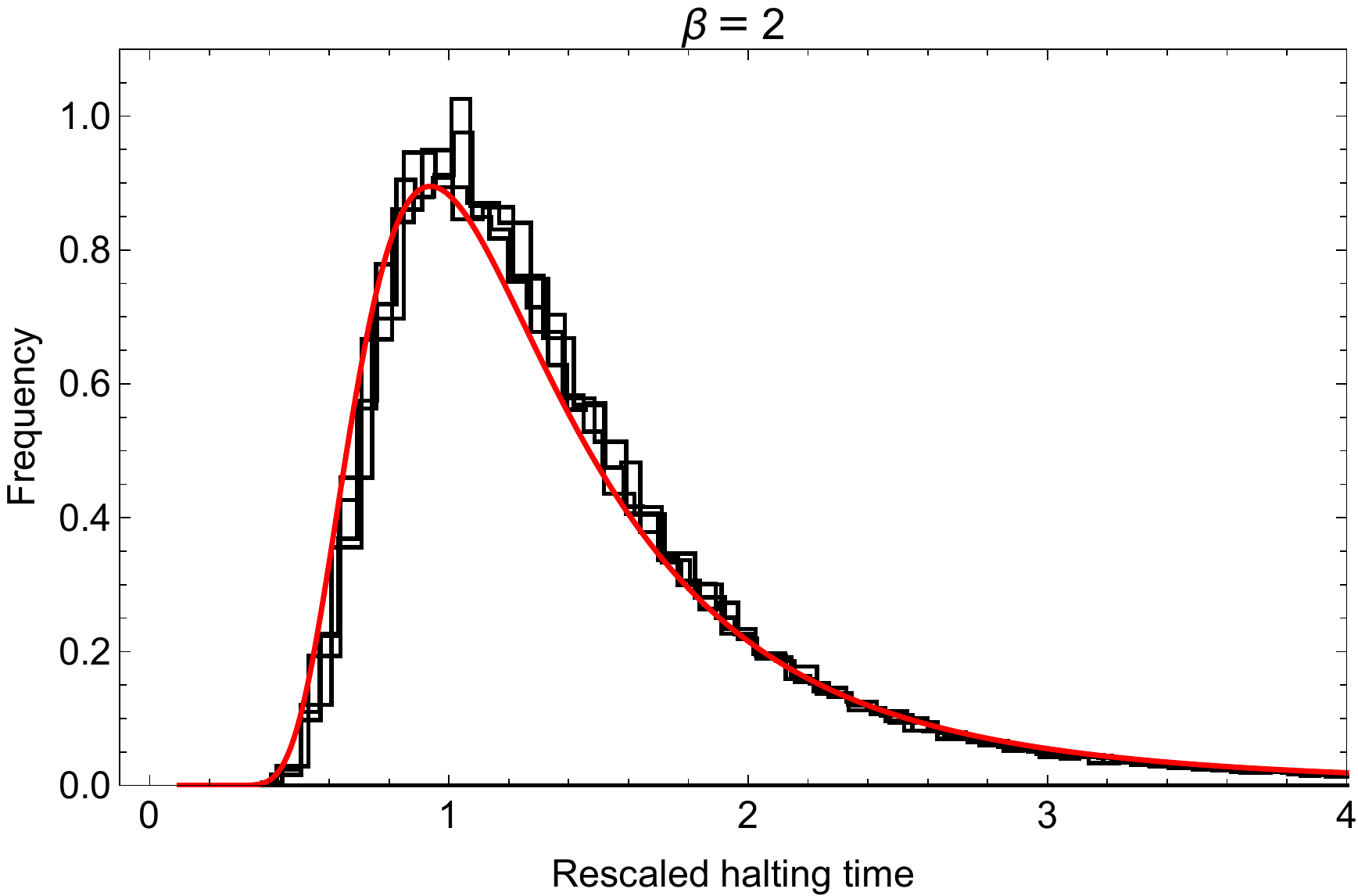}}
\caption{\label{f:QRdemon} A demonstration of Theorem~\ref{t:main} and \eqref{e:main-correction} for the QR algorithm. (a) The rescaled halting times following \eqref{e:main-correction} for LOE and BE for $d =1/2$ and $d = 2/3$ plotted against $\frac{d}{dt} F_1^{\mathrm{gap}}(t)$. (b) The rescaled halting times following \eqref{e:main-correction} for LUE and CBE for $d =1/2$ and $d = 2/3$ plotted against $\frac{d}{dt} F_2^{\mathrm{gap}}(t)$.}
\end{figure}

\begin{figure}[tbp]
\subfigure[]{\includegraphics[width=.49\linewidth]{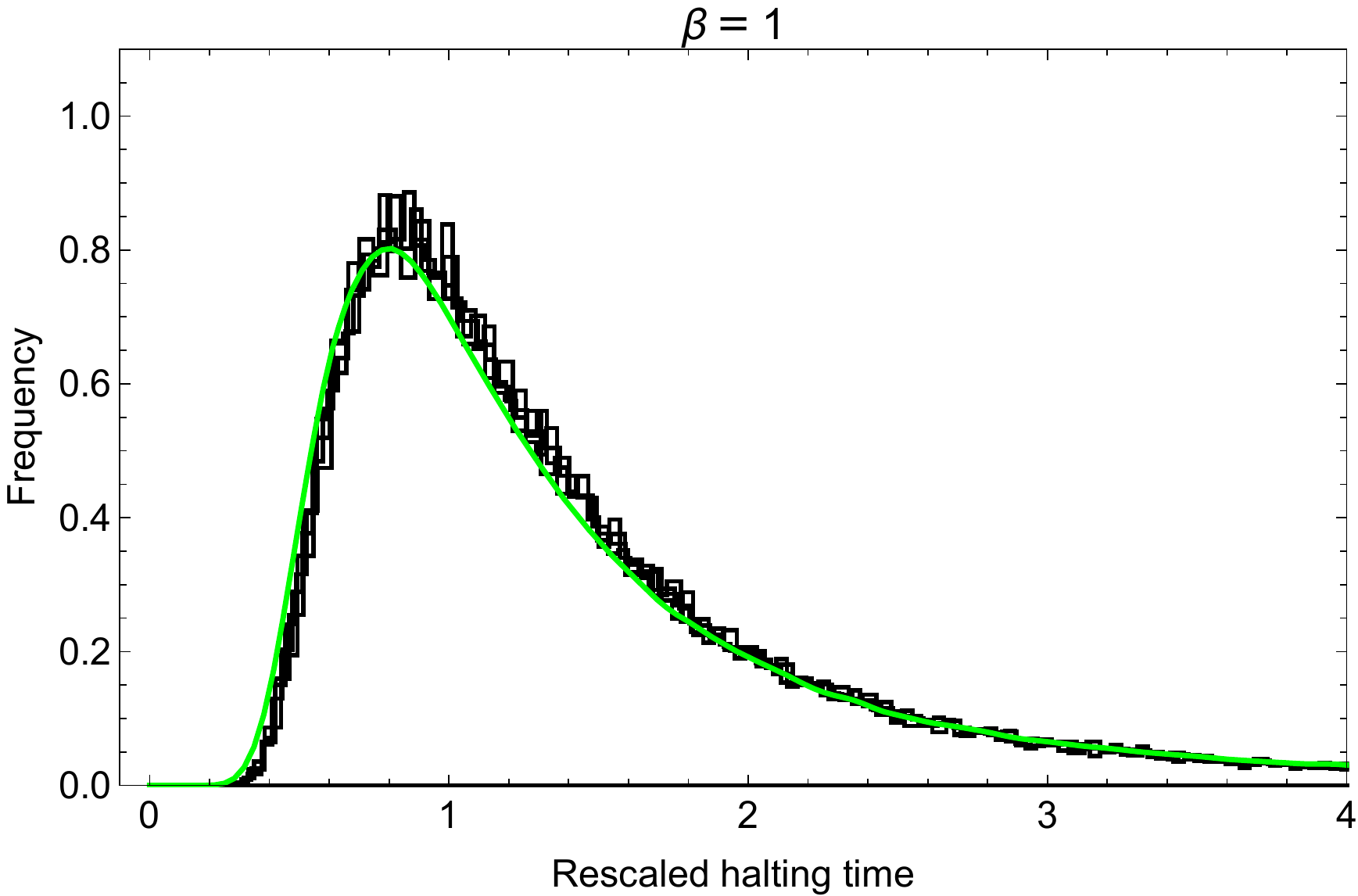}}
\subfigure[]{\includegraphics[width=.49\linewidth]{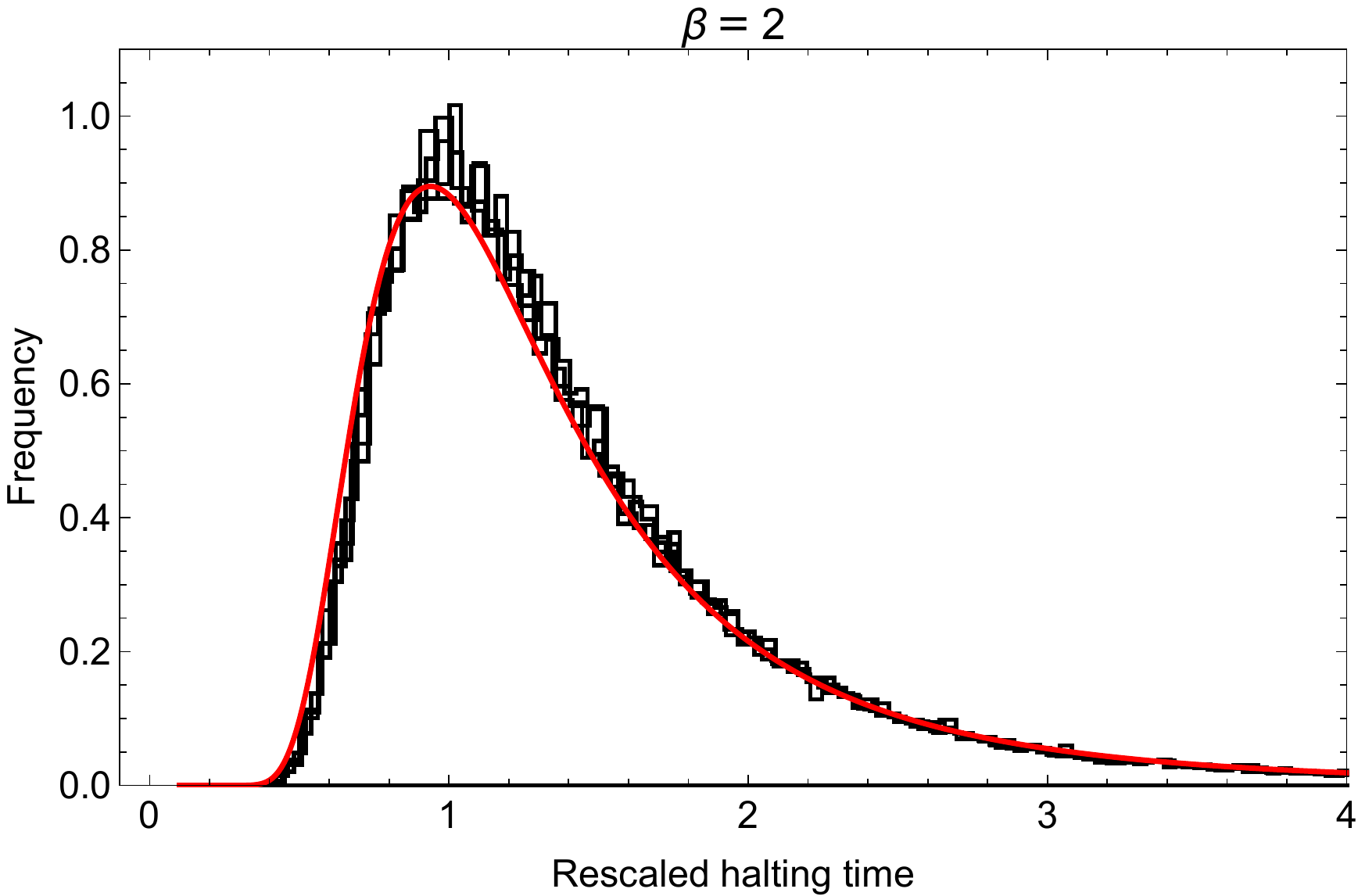}}
\caption{\label{f:IPdemon} A demonstration of Theorem~\ref{t:main} and \eqref{e:main-correction} for the inverse power method. (a) The rescaled halting times following \eqref{e:main-correction} for LOE and BE for $d =1/2$ and $d = 2/3$ plotted against $\frac{d}{dt} F_1^{\mathrm{gap}}(t)$. (b) The rescaled halting times following \eqref{e:main-correction} for LUE and CBE for $d =1/2$ and $d = 2/3$ plotted against $\frac{d}{dt} F_2^{\mathrm{gap}}(t)$.}
\end{figure}

\begin{figure}[tbp]
\subfigure[]{\includegraphics[width=.49\linewidth]{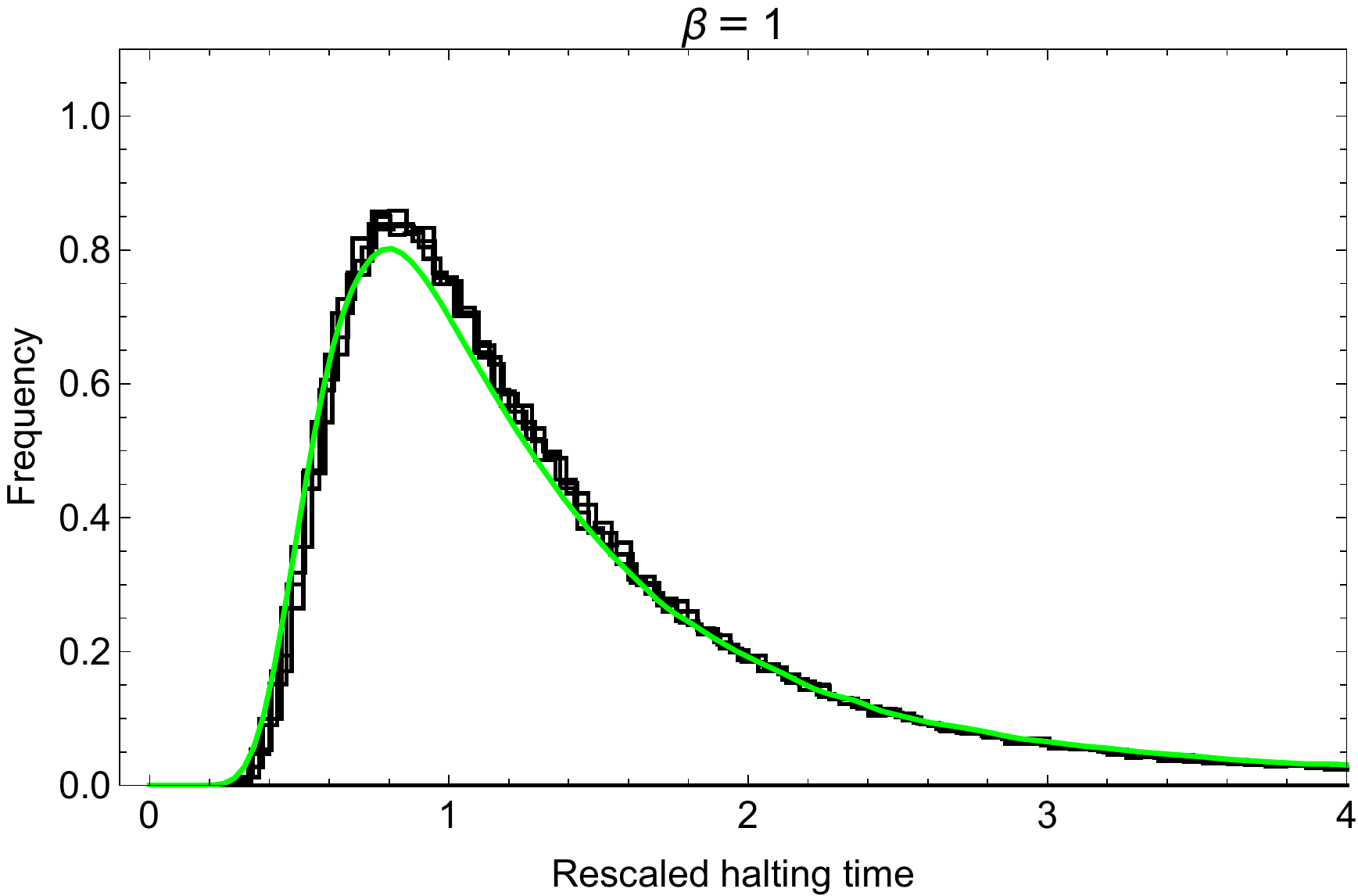}}
\subfigure[]{\includegraphics[width=.49\linewidth]{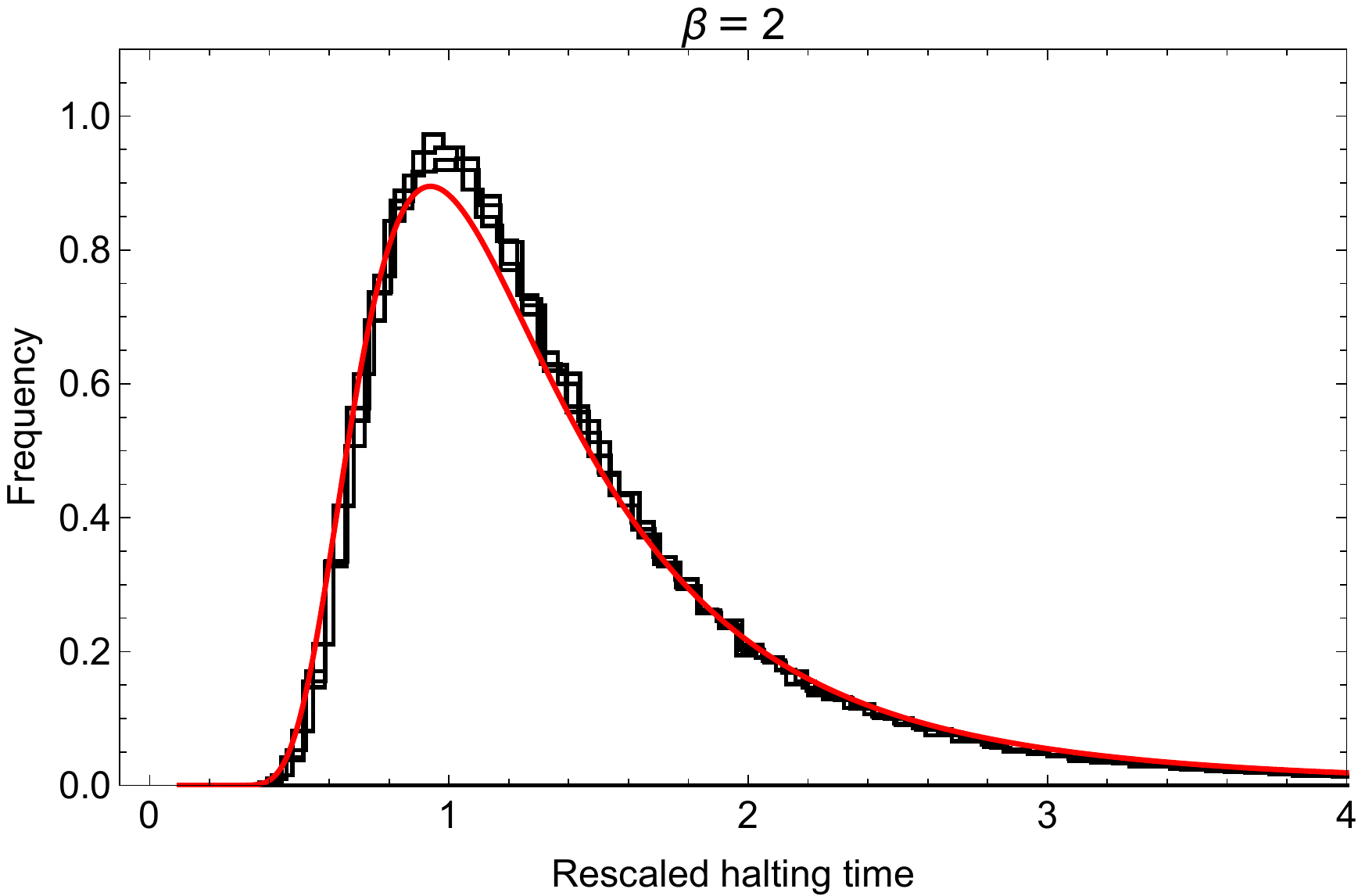}}
\caption{\label{f:Pdemon} A demonstration of Theorem~\ref{t:main} and \eqref{e:main-correction} for the power method. (a) The rescaled halting times following \eqref{e:main-correction} for LOE and BE for $d =1/2$ and $d = 2/3$ plotted against $\frac{d}{dt} F_1^{\mathrm{gap}}(t)$. (b) The rescaled halting times following \eqref{e:main-correction} for LUE and CBE for $d =1/2$ and $d = 2/3$ plotted against $\frac{d}{dt} F_2^{\mathrm{gap}}(t)$.}
\end{figure}

Finally, in Figure~\ref{f:QRdef} we show the statistics of the time of first deflation, as defined in the introduction, for LOE and BE when $d = 2$. This demonstrates universality for the time of first deflation but the limiting distribution (whatever it may be!) is clearly distinct from both histograms in Figure~\ref{f:QRdeflation} and the limiting distribution in Theorem~\ref{t:main}.   And so, computing the limiting distribution for the rescaled time of first deflation requires information about much more than just the $(N-1)$-deflation time.

\begin{figure}[tbp]
\centering
\includegraphics[width = .49\linewidth]{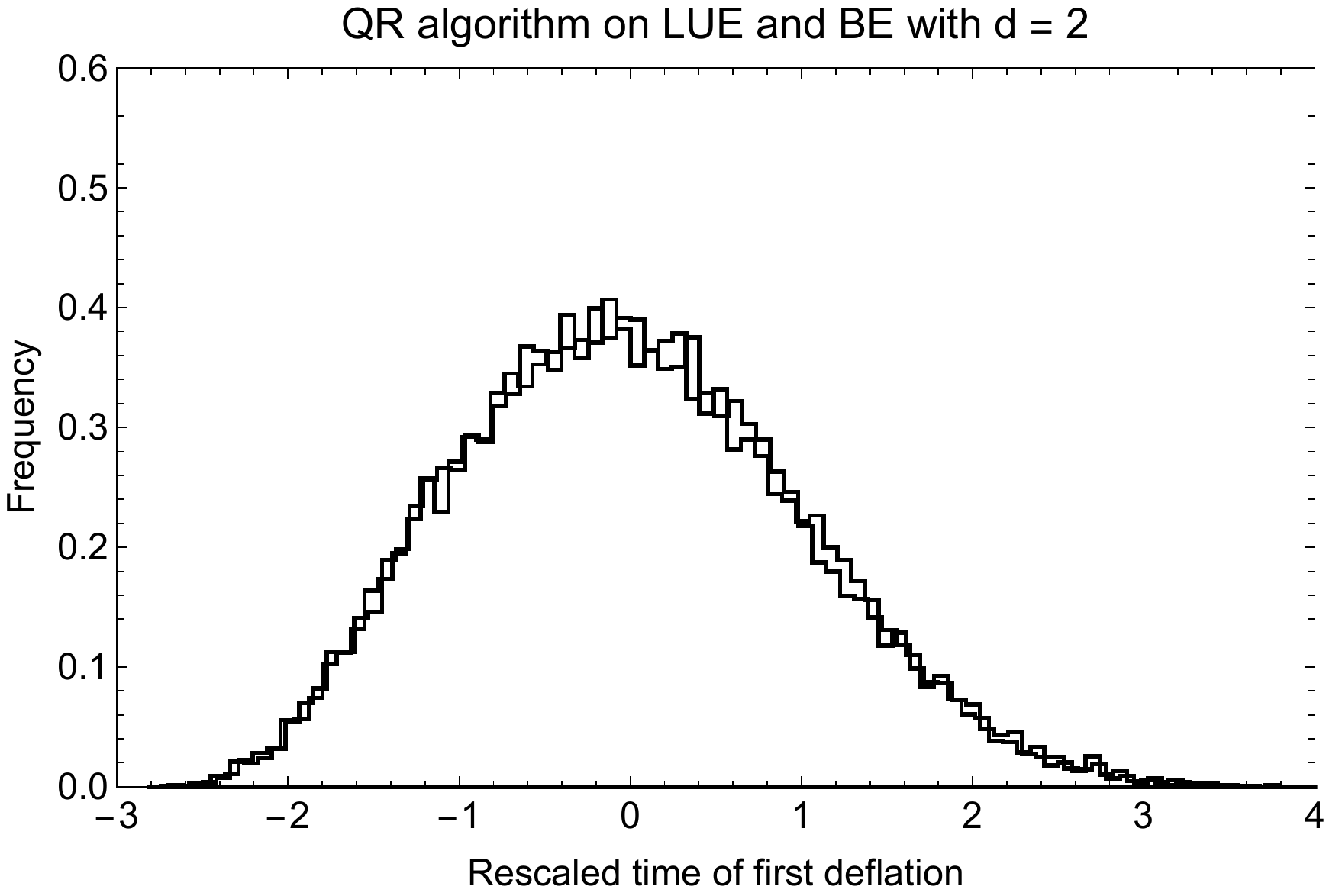}
\caption{\label{f:QRdef} Rescaled histograms (normalized to mean zero and variance one) for the time of first deflation in the QR algorithm applied to LOE and BE matrices.  We choose $d = 2$.  This distribution is much more symmetric than the distributions in Figure~\ref{f:QRdeflation} and Theorem~\ref{t:main} (displayed in Figure~\ref{f:QRdemon}).}
\end{figure}

\section{Fundamentals of the algorithms}\label{sec:alg}

Here we discuss the QR algorithm and power/inverse power methods.  We derive explicit formulae to analyze the halting times of the algorithms.

\subsection{The power and inverse power methods}\label{sec:Power}
Let $Y_1,Y_2,Y_3,\ldots,$
    be a sequence of independent, real, mean-zero, and variance-one random variables.  The power and inverse power methods with random starting are given in Algorithms~\ref{A:Power} and \ref{A:IP}.
\newcommand{\lold}{\lambda_{\mathrm{old}}}
\newcommand{\vold}{v_{\mathrm{old}}}

  \begin{algorithm}[tbp]
  Input: {$H$ and $\epsilon > 0$}\\
  Output: {$\lambda$: an approximation of the top eigenvalue of $H$ to the order of $\epsilon$}
  \begin{algorithmic}
  \STATE{set $\vold = (Y_1,Y_2,\ldots,Y_N)^T$;}
  \STATE{set $\vold = \vold/\|\vold\|_2$}
  \STATE{set $\lold = \infty$;}
  \STATE{set $v = H \vold$;}
  \STATE{set $\lambda = \langle v, \vold \rangle$;}
  \WHILE{$|\lambda-\lold| > \epsilon^2$}
  \STATE{set $v = v/\|v\|_2$;}
  \STATE{set $\vold = v$;}
  \STATE{set $\lold = \lambda$;}
  \STATE{set $v = Hv$;}
  \STATE{set $\lambda = \langle v, \vold \rangle$;}
  \ENDWHILE
  \RETURN $\lambda$
\end{algorithmic}
\caption{\label{A:Power} The power method.}
\end{algorithm}

  The power method (see Algorithm~\ref{A:Power} above) is halted when successive approximations have a difference that is less than $\epsilon^2$.  Our analysis reveals (see Proposition~\ref{p:main} and Remark~\ref{r:IP-true}) that typically $|\lambda - \lambda_{\mathrm{old}}|$ is less than the true error $|\lambda - \lambda_N|$ and so one has to run until the difference is  $\epsilon^2$.  Similarly, the inverse power method is given by Algorithm~\ref{A:IP} below where we use the convention $0^{-1} = \infty$.
  \begin{algorithm}[tbp]
  Input: {$H$ and $\epsilon > 0$}\\
  Output: {$\lambda$: an approximation of the smallest eigenvalue of $H$ to the order of $\epsilon$}
  \begin{algorithmic}
    \STATE{set $\vold = (Y_1,Y_2,\ldots,Y_N)^T$;}
    \STATE{set $\vold = \vold/\|\vold\|_2$}
    \STATE{set $\lold = 0$;}
  \STATE{set $v = H^{-1} \vold$; }
  \STATE{set $\lambda = \langle v, \vold \rangle^{-1}$}
  \WHILE{$|\lambda^{-1}-\lold^{-1}| > \epsilon^2$}
  \STATE{set $v = v/\|v\|_2$;}
  \STATE{set $\vold = v$;}
  \STATE{set $\lold = \lambda$;}
  \STATE{set $v = H^{-1}v$;}
  \STATE{set $\lambda = \langle v, \vold \rangle^{-1}$;}
\ENDWHILE
\RETURN $\lambda$
\end{algorithmic}
\caption{\label{A:IP} The inverse power method.}
\end{algorithm}

Let $H = U \Lambda U^*$, $U = (u_1, u_2, \ldots, u_N)$ be a spectral decomposition for the matrix $H$.  A random unit vector is given by
\begin{align*}
v = Y/\|Y\|_2, \quad Y = (Y_1,Y_2,\ldots,Y_N)^T,
\end{align*}
for the given random variables $Y_j$. With the inverse power method, at each iteration, $t = 1,2,3,\ldots$ we have
\begin{align*}
\lambda_{\mathrm{IP}}(t) = \frac{\langle H^{-t} v, H^{-t} v \rangle}{\langle H^{-t} v, H^{-t-1} v \rangle} = \frac{\ds \sum_{n=1}^N \lambda_n^{-2t} \beta_n^2 }{\ds \sum_{n=1}^N \lambda_n^{-2t-1} \beta_n^2 }, \quad \beta_n = |\langle v,u_n\rangle|, \quad \lambda_{\mathrm{IP}}(t) \to \lambda_1.
\end{align*}
For the power method we have
\begin{align*}
\lambda_{\mathrm{P}}(t) = \frac{\langle H^{t} v, H^{t+1} v \rangle}{\langle H^{t} v, H^{t} v \rangle} = \frac{\ds \sum_{n=1}^N \lambda_n^{2t+1} \beta_n^2 }{\ds \sum_{n=1}^N \lambda_n^{2t} \beta_n^2 }, \quad \beta_n = |\langle v,u_n\rangle|, \quad \lambda_{\mathrm{P}}(t) \to \lambda_N.
\end{align*}

\subsubsection{The halting time}

We define the halting time for the inverse power method as
\begin{align}\label{e:halt-invpow}
T_{\mathrm{IP},\epsilon}(H,v) = \inf\{t : |\lambda^{-1}_{\mathrm{IP}}(t) - \lambda^{-1}_{\mathrm{IP}}(t+1)| \leq \epsilon^2 \}.
\end{align}
Provided that the smallest eigenvalue of $H$ is order $1$, this halting condition will give the same order of approximation in $\epsilon$ as a possibly more natural condition $\inf\{t : |\lambda_{\mathrm{IP}}(t) - \lambda_{\mathrm{IP}}(t+1)| \leq \epsilon \}$.  We choose \eqref{e:halt-invpow} for convenience and show it is sufficient.  Similarly, the halting time for the power method is
\begin{align}\label{e:halt-pow}
T_{\mathrm{P},\epsilon}(H,v) = \inf\{t : |\lambda_{\mathrm{P}}(t) - \lambda_{\mathrm{P}}(t+1)| \leq \epsilon^2 \}.
\end{align}

Define the function
\begin{align*}
E_{\mathrm{IP}}(t) = \lambda^{-1}_{\mathrm{IP}}(t+1) - \lambda^{-1}_{\mathrm{IP}}(t) = \frac{\ds \sum_{n=1}^N \lambda_n^{-2t-3} \beta_n^2 }{\ds \sum_{n=1}^N \lambda_n^{-2t-2} \beta_n^2 } - \frac{\ds \sum_{n=1}^N \lambda_n^{-2t-1} \beta_n^2 }{\ds \sum_{n=1}^N \lambda_n^{-2t} \beta_n^2 }
\end{align*}
Using the notation $\delta_n = \lambda_1^2/\lambda^2_n \leq 1$, $\nu_n = \beta_n^2/\beta_1^2$, we have
\begin{align}\label{e:EIP}
\begin{split}
E_{\mathrm{IP}}(t)
& = \underbrace{\frac{\ds  \sum_{n=2}^N (1-\delta_n)(\lambda_1^{-1}- \lambda_n^{-1})\delta_n^t \nu_n }{\ds \left(\sum_{n=1}^N \delta_n^{t} \nu_n \right) \left( \sum_{n=1}^N \delta_n^{t+1} \nu_n \right) }}_{E_{\mathrm{IP},0}(t)}\\
& + \underbrace{\frac{\ds  \left(\sum_{n=2}^N \delta_n^{t+1} \lambda_n^{-1} \nu_n \right)\left(\sum_{n=2}^N \delta_n^{t} \nu_n \right)- \left(\sum_{n=2}^N \delta_n^{t+1} \nu_n \right)\left(\sum_{n=2}^N \lambda_n^{-1}\delta_n^{t} \nu_n \right) }{\ds \left(\sum_{n=1}^N \delta_n^{t} \nu_n \right) \left( \sum_{n=1}^N \delta_n^{t+1} \nu_n \right) }}_{E_{\mathrm{IP},1}(t)}.
\end{split}
\end{align}
Note that
\begin{align*}
E_{\mathrm{IP},0}(t) = \lambda_1^{-1} \frac{\ds  \sum_{n=2}^N (1-\delta^{1/2}_n)^2(\delta^{1/2}_n + 1)\delta_n^t \nu_n }{\ds \left(\sum_{n=1}^N \delta_n^{t} \nu_n \right) \left( \sum_{n=1}^N \delta_n^{t+1} \nu_n \right) }.
\end{align*}

\begin{remark}\label{r:power}
We focus on the inverse power method here.  There is an anlogous function $E_{\mathrm{P}}(t)$ for the power method which can be found through the mapping $\lambda_j \to \lambda_{j}^{-1}$.  And so, if we can estabilish properties of $E_{\mathrm{IP}}(t)$ under assumptions on $H$ that $H^{-1}$ also satifies, the properties extend to $E_{\mathrm{P}}(t)$.
\end{remark}

\subsection{The QR (eigenvalue) algorithm} \label{sec:QR}
 Unlike the power and inverse power methods, the convergence criterion for the QR algorithm (without shifts) is much more subtle even though convergence is guaranteed for the matrices we consider \cite{Francis1961}.  We consider a general error control function $f(H) \geq 0$, see below. The basis of the algorithm is the QR factorization of a non-singular matrix.  We use $(Q,R) = \mathrm{QR}(H)$ ($H = QR$) to denote this factorization where $Q$ is unitary and $R$ is upper-triangular with positive diagonal entries.  The QR factorization can be found via the modified Gram--Schmidt procedure or Householder reflections, for example.  It is unique when it exists.  The QR algorithm is given by the following steps:\\
\begin{algorithm}[H]
Input: {$H$ and $\epsilon > 0$}\\
Output: {An approximation of the spectrum of $H$}
\begin{algorithmic}
\STATE{set $X=H$;}
\WHILE{$f(X) > \epsilon$}
\STATE{set $(Q,R) = \mathrm{QR}(X)$;}
\STATE{set $X = RQ$}
\ENDWHILE
\RETURN $[X_{11},X_{22},\ldots,X_{NN}]^T$
\end{algorithmic}
\caption{\label{A:QR} The QR algorithm.}
\end{algorithm}
\vspace{.1in}

Provided that $f$ is suitably chosen, (a subset of) the diagonal entries of $X$ will be an approximation of eigenvalues of $H$.  We develop a more analytically tractable description of the QR algorithm.  For a positive-definite matrix $H$, let $H^t$ denote its $t$th power, $t \geq 0$. Define $Q(t)$, $R(t)$ and $X(t)$ via
\begin{align*}
(Q(t),R(t)) &= \mathrm{QR}(H^t),\\
H^t &= Q(t)R(t),\\
X(t) & = Q^*(t) H Q(t).
\end{align*}
For the QR algorithm we are interested in $t \in \mathbb N$ but for additional remarks we want to consider $t \geq 0$.  And so, it is important to note that $Q(t)$ and $R(t)$ are infinitely differentiable matrix-valued functions of $t$.

It is well known that $X(n)$, $n =0,1,2,\ldots$ gives the iterates $X_n$ of the QR algorithm, with, of course, $X(0) = X_0 = H$.  For the convenience of the reader, we provide the following standard proof.
\begin{lemma}
For all $n \in \mathbb N$, $X(n) = X_n$.
\end{lemma}
\begin{proof}
Using induction, the QR algorithm is described as
\begin{align*}
X_0 &= H = Q_0R_0,\\
X_1 &= R_0Q_0 = Q_1 R_1 = Q_0^* H Q_0,\\
X_2 &= R_1Q_1 = Q_1^* X_1 Q_1 = Q_1^* Q_0^* H Q_0 Q_1,\\
& \vdots\\
X_n &= Q^*_{n-1} \cdots Q_{0}^* H Q_0 \cdots Q_{n-1}.
\end{align*}
Then we consider the QR factorization of $H^n$
\begin{align*}
H^n = Q(n)R(n) & = (Q_0R_0)^n,\\
& = Q_0 (R_0Q_0)^{n-1} R_0,\\
& = Q_0 Q_1 (R_1 Q_1)^{n-2} R_1 R_0,\\
& \vdots\\
& = Q_0 \cdots Q_{n-1} R_{n-1} \cdots R_0.
\end{align*}
It then follows that $Q(n) = Q_0 Q_1 \cdots Q_{n-1}$ by the uniqueness of the QR factorization.  Therefore $X(n) = X_n$.
\end{proof}

Let $H = V \Lambda V^*$, $\Lambda = \diag(\lambda_1,\ldots, \lambda_N)$ be a\footnote{Note that $V$ is not uniquely defined.  Furthermore, if the spectrum is not simple then $V$ is not even uniquely defined modulo phases. } spectral decomposition of $H$.  Then define $U(t) = Q^*(t) V$ so that $X(t) = U(t) \Lambda U^*(t)$.  We first compute $U_{Nn}(t)$, $n = 1,2,\ldots,N$ by considering ($e_j$ is the $j$th canonical basis column vector and $U(0) = V$)
\begin{align*}
U(t) &= Q^*(t) U(0),\\
e_N^TU(t) &= e_N^T R(t) H^{-t} U(0),\\
&= R_{NN} e_N^T U(0) \Lambda^{-t},\\
U_{Nn}(t) &= e_N^T U(t) e_n = R_{NN} e_N^T U(0) \Lambda^{-t} e_n,\\
& = R_{NN} U_{Nn}(0)\lambda_n^{-t}.
\end{align*}
And so, to determine $R_{NN} > 0$, we sum over $n$ and use the normalization of the rows of $U(t)$:
\begin{align*}
R_{NN} = \left( \sum_{n=1}^N \lambda_n^{-2t} |U_{Nn}(0)|^2 \right)^{-1/2}.
\end{align*}

When it comes to the choice of the function $f(X)$ in Algorithm~\ref{A:QR}, we first give two options that we do not analyze but are of great interest:
\begin{itemize}
  \item Compute the entire spectrum: $f(X) = \|X - \diag(X)\|_{\mathrm{F}}$.  Here $\diag(X)$ is a diagonal matrix containing just the diagonal of $X$ and $\|\cdot\|_{\mathrm{F}}$ is the Frobenius norm.
  \item Deflation\footnote{Here $X(i:j,l:k)$ refers to the submatrix containing entries in rows $i$ through $j$ and columns $l$ through $k$.}:

  \begin{align*}
  f(X) = \min_{ 1\leq k \leq N-1} \|X(k+1:N,1:k)\|_2.
  \end{align*}
\end{itemize}

For our purposes here we choose $f(X)$ as
\begin{align*}
f(X) = \sqrt{\sum_{n=1}^{N-1} |X_{Nn}|^2}.
\end{align*}
This is the sum of the off-diagonal entries in the last row of $X$.  And so, if $f(X)$ is small then $X_{NN}$ is close to an eigenvalue of $X$.  Continuing,
\begin{align*}
 E_{\mathrm{QR}}(t) : = f(X(t)) & = \sum_{n=1}^{N-1} |X_{Nn}(t)|^2 = \sum_{n=1}^{N-1} X_{Nn}(t)X_{nN}(t) = [X^2(t)]_{NN} - X_{NN}^2(t)\\
& = \sum_{n=1}^N \lambda_n^2 |U_{Nn}(t)|^2 - \left(\sum_{n=1}^N \lambda_n |U_{Nn}(t)|^2 \right)^2\\
& = \frac{\ds \sum_{n=1}^N \lambda_n^{-2t+2} |U_{Nn}(0)|^2}{\ds \left( \sum_{n=1}^N \lambda_n^{-2t} |U_{Nn}(0)|^2 \right)} - \frac{\ds \left(\sum_{n=1}^N \lambda_n^{-2t+1} |U_{Nn}(0)|^2 \right)^2}{\ds \left( \sum_{n=1}^N \lambda_n^{-2t} |U_{Nn}(0)|^2 \right)^2}.
\end{align*}

\begin{remark}
It is worth emphasizing that $X(t)$, the interpolation of the QR iterates $\{X_n\}$, is the solution of a nonlinear differential equation \cite{DeiftEigenvalue}.  Furthermore, in the real symmetric case, this is generically a system in $2[N^2/4]$ variables that is Hamiltonian and completely integrable.  The eigenvalues of $X(0) = H$, constitute $N$ of the $[N^2/4]$ integrals of the motion, \emph{i.e.} the flow is, in particular, isopectral.  The equations of motion are given by
\begin{align*}
\frac{dX}{dt} = [P(\log X),X], \quad P(Y) = Y_-^T - Y_-,
\end{align*}
where $Y_-$ is the (strictly) lower-triangular part of $Y$.  The Hamiltonian is given by $H(X) = \tr (X (\log X -1))$.  See \cite{DeiftEigenvalue} and \cite{WatkinsIsospectral} for more details.
\end{remark}

\subsubsection{The halting time}  We define the halting time for the QR algorithm as
\begin{align*}
  T_{\mathrm{QR},\epsilon}(H) = \inf \{ t: E_{\mathrm{QR}}(t) \leq \epsilon^2 \}.
\end{align*}
Note that we do not assume here that $t$ is an integer.  The ``true'' halting time for the QR algorithm is $\lceil T_{\mathrm{QR},\epsilon}(H) \rceil$ but it will turn out that this has the same limiting distribution as $ T_{\mathrm{QR},\epsilon}(H)$.

The first step in the analysis of the QR algorithm is to write $E_{\mathrm{QR}}(t)$ as a sum of two positive parts, as follows.  Define for $n \geq 1$
\begin{align*}
\delta_n = \lambda^2_1/\lambda^2_n, \quad \Delta_n = \lambda_n - \lambda_1, \quad \beta_n = |U_{Nn}(0)|, \quad \nu_n = \beta_n^2/\beta^2_1.
\end{align*}
Then
\begin{align*}
E_{\mathrm{QR}}(t) =  \frac{\ds \sum_{n=1}^N \lambda_n^{2} \delta_n^{t} \beta_n^2}{\ds \left( \sum_{n=1}^N \delta^{t} \beta_n^2 \right)} - \frac{\ds \left(\sum_{n=1}^N \lambda_n \delta_n^{t} \beta_n^2 \right)^2}{\ds \left( \sum_{n=1}^N \delta_n^{t} \beta_n^2 \right)^2}.
\end{align*}
It is clear that $\delta_1 = 1 \geq \delta_n$ for all $n$ and we isolate this term:
\begin{align*}
E_{\mathrm{QR}}(t) &=  \frac{\left(\ds \sum_{n=1}^N \lambda_n^{2} \delta_n^{t} \beta_n^2 \right) \left(\ds  \sum_{n=1}^N \delta_n^{t} \beta_n^2 \right) - \ds \left(\sum_{n=1}^N \lambda_n \delta_n^{t} \beta_n^2 \right)^2}{\ds \left( \sum_{n=1}^N \delta_n^{t} \beta_n^2 \right)^2}\\
&= \frac{\left(\ds \lambda_1^2 + \sum_{n=2}^N \lambda_n^{2} \delta_n^{t} \nu_n \right) \left(1 + \ds  \sum_{n=2}^N \delta_n^{t} \nu_n \right)-\ds \left(\lambda_1 + \sum_{n=2}^N \lambda_n \delta_n^{t} \nu_n \right)^2}{\ds \left( 1 + \sum_{n=2}^N \delta_n^{2t} \nu_n \right)^2}\\
& = \underbrace{\frac{\ds \sum_{n=2}^N \Delta_n^2\delta_n^{t} \nu_n}{\ds \left( 1 + \sum_{n=2}^N \delta_n^{t} \nu_n \right)^2}}_{:= E_{\mathrm{QR},0}(t)}+\underbrace{\frac{\left(\ds \sum_{n=2}^N \lambda_n^{2} \delta_n^{t} \nu_n \right) \left(\ds  \sum_{n=2}^N \delta_n^{t} \nu_n \right) - \ds \left(\sum_{n=2}^N \lambda_n \delta_n^{t} \nu_n \right)^2}{\ds \left(1+ \sum_{n=2}^N \delta_n^{t} \nu_n \right)^2}}_{:=E_{\mathrm{QR},1}(t)}.
\end{align*}
Heuristically, $E_{\mathrm{QR},1}(t)$ is quadratic in $\delta_2^t$ and $E_{\mathrm{QR},0}(t)$ is not.  Therefore $E_{\mathrm{QR},0}(t)$ should provide the leading order behavior of $E_{\mathrm{QR}}(t)$ as $t \to \infty$ provided that the $\nu_n$'s are not too large.  Note that by the Cauchy--Schwartz inequality, $E_{\mathrm{QR},1}(t) \geq 0$.

\section{Proofs of the main theorems} \label{sec:proofs}

In order to prove our main theorems we take the following approach.  The dynamics of the QR algorithm closely mirrors that of the so-called Toda algorithm and therefore many of the results of \cite{Deift2016} apply directly.  And to prove Theorem~\ref{t:main} for the QR algorithm we almost exclusively simply quote results from \cite{Deift2016}.  To prove Theorem~\ref{t:main} for the power and inverse power methods, we discuss the calculations in more detail.

For convenience let $\epsilon = N^{-\alpha/2}$. Then Condition~\ref{cond:scaling} takes the form
\begin{align*}
 \alpha \geq 10/3 + \sigma,
\end{align*}
with $\sigma > 0$ and fixed.

\subsection{Technical lemmas}

We begin by modifying the technical lemmas from \cite{Deift2016} as our formulae now depend on the ratio of eigenvalues as opposed to their differences in \cite{Deift2016}. The main fact is that if a matrix $H$ satisfies Condition~\ref{cond:rigidity} with $0 < s < 1/5$ then so does $\log H$ with quantiles $\hat \gamma_n = \log \gamma_n$ provided $N$ is sufficiently large.  Indeed for $0 < s < \sigma/40$
\begin{align*}
 |\log \lambda_n - \log \gamma_n| &\leq \frac{1}{\lambda_1} |\lambda_n - \gamma_n| \leq \frac{1}{\gamma_1 - N^{-2/3+s/2}} |\lambda_n - \gamma_n| \\
 & \leq \frac{2}{\lambda_-}|\lambda_n - \gamma_n| \leq N^{s/4}|\lambda_n - \gamma_n|,
\end{align*}
for $N$ sufficiently large.  Applying Condition~\ref{cond:rigidity}(5) with $s$ replaced by $s/2$, we conclude that for $N$ sufficiently large $|\log \lambda_n - \hat \gamma_n| \leq N^{-2/3+s/2}(\min\{n,N-n+1\})^{-1/3}$ for all $n$. Concerning Condition~\ref{cond:rigidity}(3), note that for $N$ sufficiently large
\begin{align*}
\frac{\lambda_N - \lambda_n}{2 \lambda_+} \leq \frac{\lambda_N-\lambda_n}{\lambda_N } \leq \log \lambda_N - \log \lambda_n \leq \frac{\lambda_N - \lambda_n}{\lambda_1} \leq \frac{2}{\lambda_-} (\lambda_N - \lambda_n)
\end{align*}
and one then proceeds as before.  The proof of Condition~\ref{cond:rigidity}(4) is similar.

Recall the notation $\delta_n = \lambda_1^2/\lambda_n^2$ and define $I_c = \{2 \leq n \leq N: \delta_n \leq \delta_2^{1+c}\}$ for $c > 0$.
\begin{lemma}[\cite{Deift2016}]\label{l:Ic}
  Let $0 < c < 10/\sigma$.  Given Condition~\ref{cond:rigidity},  then the cardinality of $I_c^c$ is given by
  \begin{align*}
    |I_c^c| \leq N^{2s}
  \end{align*}
  for $N$ sufficiently large, where $^c$ denotes the compliment relative to $\{1,\ldots,N-1\}$.
\end{lemma}

    Recalling the notation $\nu_n = \beta_n^2/\beta_1^2$, for matrices in $\mathcal R_{N,s}$ we have $\nu_n \leq N^{2s}$ and $\sum_n \nu_n = \beta_1^{-2} \leq N^{1+s}$ because $\sum_{n=1}^N \beta_n^2 = \sum_{n=1}^N |\langle v,u_j\rangle|^2 = \|U v\|_2 = \|v\|_2$ for the unitary matrix $U$ of eigenvectors. We also have the following result.

  \begin{lemma}[\cite{Deift2016}]\label{l:estimate}
	Given Condition~\ref{cond:rigidity}, $0 < c < 10/\sigma$ and $j \leq 3$ fixed there exists an absolute constant $C$ such that
    \begin{align*}
    	 N^{-2s} \Delta_{2}^j \delta_2^t \leq \sum_{n=2}^{N} \nu_n \Delta^j_n \delta_n^t \leq C \delta_2^t \left(N^{4s} \Delta_{2}^j + N^{1+s} \delta_2^{ct} \right),
	\end{align*}
    for $N$ sufficiently large.
\end{lemma}

\subsection{Main estimates for the QR algorithm}

The steps of the proof are the following:
\begin{enumerate}[(1)]
  \item \emph{a priori} estimates on $T_{\mathrm{QR},\epsilon}$ that will hold with high probability,
  \item a lower bound on $-E_{\mathrm{QR},0}'(t)$ over a region determined in (1),
  \item finding and estimating an approximation $T^*_{\mathrm{QR},\epsilon}$ of $T_{\mathrm{QR},\epsilon}$, and
  \item establishing that $T^*_{\mathrm{QR},\epsilon}$ converges in distribution and then using (1)-(3) to show that indeed $T^*_{\mathrm{QR},\epsilon}$ is close to $T_{\mathrm{QR},\epsilon}$.
\end{enumerate}

If $\epsilon$ is sufficiently small we expect $E_{\mathrm{QR},0}(t)$ to control the convergence of the algorithm.  Consider
\begin{align*}
E_{\mathrm{QR},0}(t) = \Delta_2^2\delta_2^t\nu_2 \frac{\ds  1 + \sum_{n=3}^N \frac{\Delta_n^2}{\Delta_2^2}\left(\frac{\lambda_2}{\lambda_n}\right) \frac{\nu_n}{\nu_2}}{\ds \left( 1 + \sum_{n=2}^N \delta_n^{t} \nu_n \right)^2},
\end{align*}
and the approximation $T^*_{\mathrm{QR},\epsilon}$ of $T_{\mathrm{QR},\epsilon}$ is given by
\begin{align*}
\Delta_2^2\delta_2^{T^*_{\mathrm{QR},\epsilon}}\nu_2 &= \epsilon^2.
\end{align*}
Thus
\begin{align}
T^*_{\mathrm{QR},\epsilon} &= (\alpha \log N + 2 \log \Delta_2 + \log \nu_2)/\log \delta_2^{-1}.\label{e:t-starQR}
\end{align}
To determine how close $T^*_{\mathrm{QR},\epsilon}$ is to $T_{\mathrm{QR},\epsilon}$ we use the following relation
\begin{align}\label{e:meanval}
E_{\mathrm{QR},0}(T^*_{\mathrm{QR},\epsilon}) -E_{\mathrm{QR},0}(T_{\mathrm{QR},\epsilon}) = E_{\mathrm{QR},0}'(\eta)(T^*_{\mathrm{QR},\epsilon}-T_{\mathrm{QR},\epsilon}),
\end{align}
for some $\eta$ between $T^*_{\mathrm{QR},\epsilon}$ and $T_{\mathrm{QR},\epsilon}$.  So, we need to show that the left-hand side of \eqref{e:meanval} is small and $E_{\mathrm{QR},0}'(\eta)$ is not too small. This is accomplished by following Lemmas~\ref{l:QR-T}-\ref{l:H1} and Proposition~\ref{p:conv}.  The the proofs of these results make heavy use of Lemmas~\ref{l:Ic} and \ref{l:estimate}.

\begin{lemma}[\cite{Deift2016}, Lemma 2.1]\label{l:QR-T}
Given Condition~\ref{cond:rigidity}, the halting time $T_{\mathrm{QR},\epsilon}$ for the QR algorithm satisfies
\begin{align*}
(\alpha - 4/3 - 5s) \log N/\log \delta^{-1}_{2} \leq T_{\mathrm{QR},\epsilon} \leq (\alpha - 4/3 + 7s) \log N /\log \delta^{-1}_{2},
\end{align*}
for sufficiently large $N$.
\end{lemma}
Define the interval
\begin{align*}
L_\alpha = [(\alpha - 4/3 - 5s) \log N/\log \delta^{-1}_{2} , (\alpha - 4/3 + 9s) \log N /\log \delta^{-1}_{2}].
\end{align*}
\begin{lemma}[\cite{Deift2016}, Lemma 2.2]\label{l:DH0}
Given Condition~\ref{cond:rigidity} and $t \in L_\alpha$
\begin{align*}
-E_{\mathrm{QR},0}'(t) \geq C N^{-12s-\alpha-2/3},
\end{align*}
for sufficiently large $N$.
\end{lemma}
The next estimate is immediate from the definition of $T^*_{\mathrm{QR}, \epsilon}$
\begin{lemma}[\cite{Deift2016}, Lemma 2.3]\label{l:t-star} Given Condition~\ref{cond:rigidity}
\begin{align*}
(\alpha -4/3 -4s)\log N/\log \delta^{-1}_{2} \leq T^*_{\mathrm{QR},\epsilon} \leq (\alpha -4/3 +4s) \log N/\log\delta^{-1}_{2},
\end{align*}
for sufficiently large $N$, \emph{i.e.} $T^*_{\mathrm{QR},\epsilon} \in L_\alpha$.
\end{lemma}
\begin{lemma}[\cite{Deift2016}, Lemma 2.4]\label{l:H0Tstar}
Given Conditions~\ref{cond:lowergap} and \ref{cond:rigidity}
\begin{align*}
|E_{\mathrm{QR},0}(T^*_{\mathrm{QR},\epsilon}) - N^{-\alpha}| \leq CN^{-\alpha-2p+4s},
\end{align*}
for sufficiently large $N$.
\end{lemma}
\begin{lemma}[\cite{Deift2016}, (2.6)]\label{l:H1}
Given Conditions \ref{cond:rigidity}, for $t \in L_\alpha$
\begin{align*}
|E_{\mathrm{QR},1}(t)| \leq N^{-2\alpha + 8/3+18s} \leq N^{-\alpha - 2/3- \sigma/2},
\end{align*}
for sufficiently large $N$.
\end{lemma}

\begin{proposition}\label{p:conv}
Given Conditions~\ref{cond:uppergap} and \ref{cond:rigidity} for $\sigma$ and $p$ fixed with $s$ sufficiently small (depending on $\sigma$ and $p$)
\begin{align*}
N^{-2/3}|T^*_{\mathrm{QR}, \epsilon} - T_{\mathrm{QR}, \epsilon}| \leq C N^{-2p + 16s}
\end{align*}
for $N$ sufficiently large.
\end{proposition}
\begin{proof}
We use \eqref{e:meanval} to estimate the difference (for some $\eta$ between $T_{\mathrm{QR},\epsilon}$ and $T^*_{\mathrm{QR},\epsilon})$ and apply Lemmas~\ref{l:QR-T}, \ref{l:DH0}, \ref{l:t-star}, \ref{l:H0Tstar} and \ref{l:H1} to find
\begin{align*}
|T^*_{\mathrm{QR},\epsilon}&-T_{\mathrm{QR},\epsilon}| \\
&\leq  \frac{|E_{\mathrm{QR},0}(T^*_{\mathrm{QR},\epsilon}) -E_{\mathrm{QR},0}(T_{\mathrm{QR},\epsilon})|}{ |E_{\mathrm{QR},0}'(\eta)|} \leq \frac{|E_{\mathrm{QR},0}(T^*_{\mathrm{QR},\epsilon}) -N^{-\alpha} + E_{\mathrm{QR},1}(T_{\mathrm{QR},\epsilon})|}{ \max_{\eta\in L_\alpha} E_{\mathrm{QR},0}'(\eta)}\\
& \leq \frac{|E_{\mathrm{QR},0}(T^*_{\mathrm{QR},\epsilon}) -N^{-\alpha}| + |E_{\mathrm{QR},1}(T_{\mathrm{QR},\epsilon})|}{ \max_{\eta\in L_\alpha} E_{\mathrm{QR},0}'(\eta)}\\
& \leq C N^{12s + \alpha + 2/3} \left( N^{-\alpha - 2p + 4s}  + N^{-\alpha - 2/3- \sigma/2}  \right)
\end{align*}
Using the assumption that $\alpha > 10/3$ the proposition follows.
\end{proof}

Thus far, no estimates had any probabalistic input.  We now introduce the probabilistic considerations needed to prove our main theorem for the QR algorithm.

\begin{theorem}\label{t:QR-tech}
Let $H$ be an SCM.   For $\alpha \geq 10/3 +\sigma$, $\sigma > 0$
  \begin{align*}
  F_\beta^{\mathrm{gap}}(t) = \lim_{N \to \infty} \mathbb P \left( \frac{T_{\mathrm{QR},\epsilon}}{(\alpha/2 - 2/3) \lambda_-^{1/3} d^{1/2} N^{2/3} \log N}  \leq t\right).
  \end{align*}
\end{theorem}
\begin{proof}
We first prove that the following three random variables converge to zero in probability:
\begin{align*}
&N^{-2/3} |T^*_{\mathrm{QR}, \epsilon} - T_{\mathrm{QR}, \epsilon}|, \quad \left|\frac{T^*_{\mathrm{QR}, \epsilon}}{N^{2/3} \log N} -  \frac{\alpha - 4/3}{ N^{2/3}\log \delta_2^{-1}}\right|, \text{ and }\\
&\left|\frac{\alpha/2 - 2/3}{ N^{2/3} \Delta_2} -  \frac{\alpha - 4/3}{ \lambda_- N^{2/3}\log \delta_2^{-1}}\right|.
\end{align*}
The proof for the first random variable follows \cite[Lemma 3.1]{Deift2016} and requires the specific use of Condition~\ref{cond:uppergap}, the proof for the second follows \cite[Lemma 3.4]{Deift2016}.  For the last, we write $\lambda_j = \lambda_- + N^{-2/3} \xi_j$, $j = 1,2$ where $(\xi_1,\xi_2)$ converges jointly in distribution.  Let $B_R$ be the event where $\|(\xi_1,\xi_2)\|_2 \leq R$.  Given $B_R$, consider
\begin{align*}
X_N := \frac{\alpha/2 - 2/3}{ N^{2/3} \Delta_2} -  \lambda_+\frac{\alpha - 4/3}{ N^{2/3}\log \delta_2^{-1}} = \frac{\alpha/2 - 2/3}{\xi_2 - \xi_1} -  \frac{\alpha - 4/3}{ \lambda_- N^{2/3}\log \delta_2^{-1}}.
\end{align*}
Then
\begin{align*}
\frac{1}{N^{2/3}\log \delta_2^{-1}} = \frac{1}{2N^{2/3}[\log \lambda_2 - \log \lambda_1]} \approx \frac{\lambda_-}{2[\xi_2-\xi_1]},
\end{align*}
where the approximation is uniform as $N \to \infty$ (given $B_R$).  For $\delta> 0$, $s > 0$ (sufficiently small) using uniform convergence
\begin{align*}
\mathbb P \left(  X_N \geq \delta  \right) &= \mathbb P \left(  X_N \geq \delta, B_R \right) + \mathbb P \left(  X_N \geq \delta, B_R^c \right),\\
\limsup_{N \to \infty}\mathbb P \left(  X_N \geq \delta  \right) &\leq  \limsup_{N \to \infty}\mathbb P \left( B_R^c \right).
\end{align*}
Letting $R \to \infty$ we establish that $X_N$ converges to zero in probability.  Appealing to Definition~\ref{def:Fbeta} we finally have
\begin{align*}
F_\beta^{\mathrm{gap}}(t) &= \lim_{N \to \infty} \mathbb P \left( \frac{\lambda_-^{2/3}}{d^{1/2}(\xi_2 - \xi_1)} \leq t \right)\\
& = \lim_{N \to \infty} \mathbb P \left( \frac{T_{\mathrm{QR},\epsilon}}{(\alpha/2 - 2/3)\lambda_-^{1/3} d^{1/2} N^{2/3} \log N} \leq t \right).
\end{align*}

\end{proof}

\subsection{Main estimates for the power/inverse power method}

We now follow the same steps that were performed for the QR algorithm for the inverse power method.  First, we establish $E_{\mathrm{IP},1}(t) \geq 0$ as given in \eqref{e:EIP}.

Define $w_n = \delta_n^t \nu_n /\left( \sum_{n=2}^N \delta_n^t \nu_n \right)$ and use the notation $\mathbb E_w[\delta^\alpha] = \sum_{n=2}^N \delta_n^\alpha w_n$. It follows that the non-negativity of $E_{\mathrm{IP},1}(t)$ is equivalent to
\begin{align*}
\mathbb E_w[\delta^{3/2}] - \mathbb E_w[\delta] \mathbb E_w[\delta^{1/2}]\geq 0.
\end{align*}
From Jensen's inequality for concave functions
\begin{align*}
\mathbb E_w[\delta]^{1/2} \geq \mathbb E_w[\delta^{1/2}]
\end{align*}
which gives
\begin{align*}
\mathbb E_w[\delta^{3/2}] - \mathbb E_w[\delta] \mathbb E_2[\delta^{1/2}] \geq \mathbb E_w[\delta^{3/2}] - \mathbb E_w[\delta]^{3/2} \geq 0.
\end{align*}
The last inequality follows from another application of Jensen's inequality (for convex functions).

\begin{lemma}\label{l:Tip}
Given Condition~\ref{cond:rigidity}, the halting time $T_{\mathrm{IP},\epsilon}$ for the inverse power method satisfies
\begin{align*}
(\alpha - 4/3 - 5 s) \log N/\log \delta_{2}^{-1} \leq T_{\mathrm{IP},\epsilon} \leq (\alpha - 4/3 + 6 s) \log N/ \log \delta_{2}^{-1},
\end{align*}
for sufficiently large $N$.
\end{lemma}
\begin{proof}
Because $E_{\mathrm{IP},0}(t)$ and $E_{\mathrm{IP},1}(t) $ are both positive we know that if $E_{\mathrm{IP},0}(t) > \epsilon^2$ on the interval $[0,T]$ then $T_{\mathrm{IP},\epsilon} > T$. We first estimate $E_{\mathrm{IP},0}(t)$ as follows:
\begin{align*}
E_{\mathrm{IP},0}(t) \geq \lambda_1^{-1} \frac{\ds  \sum_{n=2}^N (1-\delta^{1/2}_n)^2(\delta^{1/2}_n + 1)\delta_n^t \nu_n }{\ds \left(\sum_{n=1}^N \delta_n^{t} \nu_n \right)^2 }.
\end{align*}
Then we set $t = a \log N/\log \delta_2^{-1}$ and use  Lemma~\ref{l:estimate} to estimate the denominator
\begin{align*}
\sum_{n=1}^N \delta_n^{t} \nu_n  \leq 1 + C \delta_2^{t} (N^{4s} + N^{1+s} \delta_2^{ct}) = 1 + C N^{-a} (N^{4s} + N^{1+s -ca}).
\end{align*}
To estimate the numerator we use Lemma~\ref{l:estimate} again
\begin{align}\label{e:lowerbound}
  \begin{split}
\lambda_1^{-1}&\sum_{n=2}^N (1-\delta^{1/2}_n)^2(\delta^{1/2}_n + 1)\delta_n^t \nu_n \geq \lambda_1^{-1}\sum_{n=2}^N \lambda_n^{-2} \Delta_n^2\delta_n^t \nu_n\\
& \geq \lambda_N^{-2}\lambda_1^{-1}\sum_{n=2}^N \Delta_n^2\delta_n^t \nu_n  \geq C N^{-3s-4/3-a}.
\end{split}
\end{align}
Therefore
\begin{align*}
E_{\mathrm{IP},0}(t) \geq C N^{-3s-4/3-a} \left(  1 + C N^{-a} (N^{4s} + N^{1+s -ca})  \right)^{-2}
\end{align*}
Recall that $\alpha \geq 10/3 + \sigma$, $0 < s < \sigma/40 < 1/120 < 1/5$ and assume that $a \leq \sigma/2$.  We have
\begin{align*}
  E_{\mathrm{IP},0}(t) &\geq  CN^{-3s-4/3-\sigma/2} \left(  1 + C N^{-a} (N^{4s} + N^{1+s -ca})  \right)^{-2} \\
  &\geq CN^{-5s-10/3-\sigma/2} > CN^{-10/3-\sigma + s}
\end{align*}
which is larger than $\epsilon^2 \leq N^{-10/3-\sigma}$ for sufficiently large $N$, and we conclude
\begin{align*}
  T_{\mathrm{IP},\epsilon} \geq (\sigma/2)\log N/\log \delta_2^{-1}.
  \end{align*}  For $a \geq \sigma/2$, note that
\begin{align*}
(1 + C N^{-\sigma/2} (N^{4s} + N^{1+s -c\sigma/2}))^{-1} \leq \left( \sum_{n=1}^N \delta_n^{t} \nu_n \right)^{-1} \leq 1.
\end{align*}
Now choose $c < 10/\sigma$ (cf. Lemma~\ref{l:estimate}) so that $1+s-c\sigma/2 < 0$, i.e., $c > 2(1+s)/\sigma$.  Note that as $s < \sigma/40 < 1, ~2(1+s) < 10$, such a $c$ exists.  Furthermore, as $s < \sigma/40$, it follows from the above inequality that there exists $C> 1$ such that
\begin{align}\label{e:order1}
C^{-1} \leq \left( \sum_{n=1}^N \delta_n^{t} \nu_n \right)^{-1} \leq 1,
\end{align}
for sufficiently large $N$.  Then (again for $t=a \log N /\log \delta_2^{-1}$)
\begin{align*}
  E_{\mathrm{IP},0}(t) &\geq  CN^{-4s-4/3-a}.
\end{align*}
Now for $s < \sigma/40$, we have $\alpha - 4/3 -5s > \sigma/2$, and so for $\sigma/2 < a < \alpha - 4/3 - 5s$, we again have $E_{\mathrm{IP},0}(t) \geq  CN^{-\alpha + s} \geq \epsilon^2$ for sufficiently large $N$.  This establishes the lower bound on $T_{\mathrm{IP},\epsilon}$.

To establish the upper bound on $T_{\mathrm{IP},\epsilon}$ we use the absolute boundedness of $\lambda_{1}^{-1}$ (given Condition~\ref{cond:rigidity}: see also Remark~\ref{r:quantiles}) for $c>0$, together with $\delta_n \leq 1$
\begin{align}
\lambda_1^{-1}&\sum_{n=2}^N (1-\delta^{1/2}_n)^2(\delta^{1/2}_n + 1)\delta_n^t \nu_n \leq \frac{2}{\lambda_1} \sum_{n=2}^N  \lambda_n^{-2} \Delta_n^2 \delta_n^t  \nu_n \notag\\
&\leq \frac{2}{\lambda_1^3} \sum_{n=2}^N  \Delta_n^2 \delta_n^t \nu_n \leq C \delta_2^t (N^{4s} \Delta_2^2 + N^{1+s} \delta_2^{ct}) \leq  C N^{-a} (N^{5s-4/3} + N^{1+s-ca}).\label{e:numerbound}
\end{align}
If $t = a \log N/\log \delta_2^{-1}$ with $a \geq (\alpha - 4/3 + 6s)$, then $N^{-a + 5s - 4/3} \leq N^{-\alpha - s}$.  But then $a \geq \alpha -4/3 + 6s \geq 2$, and so, taking $c = 2 (< 10/\sigma)$, $1 + s - ca \leq s - 3$.  Hence $N^{-a + 1 + s -ca} \leq N^{-\alpha - s}$.  Hence $N^{-a + 1 + s -ca} \leq N^{-\alpha - 5/3 - 5s} \leq N^{-\alpha - s}$.  Thus
\begin{align*} 
E_{\mathrm{IP},0}(t) \leq C (N^{-\alpha-s} + N^{-\alpha -5/3 - 5s }) \leq C N^{-\alpha-s}.
\end{align*}
So, for these values of $t$, $E_{\mathrm{IP},0}(t) < \epsilon ^2$ for sufficiently large $N$.  Next, we show that the same holds for $E_{\mathrm{IP},1}(t)$.  We use the estimate with $c =2$ and any $\gamma$, to obtain
\begin{align}\label{e:gamma-est}
\sum_{n=2}^N \delta_n^t \nu_n &\leq C N^{-\alpha +4/3 - \gamma s} (N^{4s} + N^{-3+s}) \leq C N^{-\alpha + 4/3 + (4-\gamma)s}
\end{align}
for $t \geq (\alpha - 4/3 + \gamma s) \log N/\log \delta_2^{-1}$ and $N$ sufficiently large.  Then using $\lambda_{n}^{-1} \leq \lambda_1^{-1} \leq C$ (given Condition~\ref{cond:rigidity}), and taking $\gamma = -5$, $E_{\mathrm{IP},1}(t) \leq C N^{-2\alpha + 8/3 + 18s}$ for  $t \geq (\alpha -4/3 -5 s) \log N/\log \delta_2^{-1}$ and $N$ sufficiently large.  Thus
\begin{align}\label{e:IP1}
E_{\mathrm{IP},1}(t) \leq C N^{-\alpha -2/3 - \sigma + 18s} \leq C N^{-\alpha -2/3 - \sigma/2}\quad \text{for} \quad  s < \sigma/40.
\end{align}
This shows that $T_{\mathrm{IP},\epsilon} \leq (\alpha - 4/3 + 6 s) \log N/\log \delta_2^{-1}$ for large $N$. 
\end{proof}
\begin{remark}
We take $\gamma = -5$, rather than $\gamma= 6$, for technical reasons, see Lemma~\ref{l:IP1est} below.
\end{remark}

Similar to the case of the QR algorithm, define
\begin{align*}
\hat L_\alpha = [(\alpha - 4/3 - 5 s) \log N/\log \delta_{2}^{-1} , (\alpha - 4/3 + 6 s) \log N/ \log \delta_{2}^{-1}].
\end{align*}

\begin{lemma}\label{l:E0prime}
  Given Condition~\ref{cond:rigidity} and $t \in \hat L_\alpha$
\begin{align*}
-E_{\mathrm{IP},0}'(t) \geq C N^{-11s - 2/3 - \alpha},
\end{align*}
for $N$ sufficiently large.
\end{lemma}
\begin{proof}
  By direct calculation
  \begin{align*}
-&E_{\mathrm{IP},0}'(t) = \frac{\ds  \sum_{n=2}^N \log \delta_n^{-1}(1-\delta_n)(\lambda_1^{-1}
- \lambda_n^{-1})\delta_n^t \nu_n }{\ds \left(\sum_{n=1}^N  \delta_n^{t} \nu_n \right) \left( \sum_{n=1}^N \delta_n^{t+1} \nu_n \right)} \\
&- \frac{\ds   \left( \sum_{n=2}^N (1-\delta_n)(\lambda_1^{-1}- \lambda_n^{-1})\delta_n^t \nu_n  \right)}{\ds \left(\sum_{n=1}^N  \delta_n^{t} \nu_n \right)^2 \left( \sum_{n=1}^N \delta_n^{t+1} \nu_n \right)^2}\\&
\times \left[  \left(\sum_{n=2}^N  \log \delta_n^{-1}\delta_n^{t} \nu_n \right) \left( \sum_{n=1}^N \delta_n^{t+1} \nu_n \right) + \left(\sum_{n=1}^N  \delta_n^{t} \nu_n \right) \left( \sum_{n=2}^N\log \delta_n^{-1} \delta_n^{t+1} \nu_n \right)\right].
\end{align*}
Then using \eqref{e:lowerbound} and $\log \delta_2^{-1} \geq C N^{-2/3-s/2}$ and keeping only the leading term
\begin{align*}
F(t)&:=\frac{\ds  \sum_{n=2}^N \log \delta_n^{-1}(1-\delta_n)(\lambda_1^{-1}- \lambda_n^{-1})\delta_n^t \nu_n }{\ds \left(\sum_{n=1}^N  \delta_n^{t} \nu_n \right) \left( \sum_{n=1}^N \delta_n^{t+1} \nu_n \right)} \\
&\geq C \log \delta_2^{-1}(1-\delta_2)(\lambda_1^{-1}- \lambda_2^{-1})\delta_2^t \nu_2 \geq C \Delta_2^3 \delta_2^t \nu_2  \geq C N^{-11s - 2/3 - \alpha},
\end{align*}
for $t \in \hat L_\alpha$.  Define $G(t)$ by $-E_{\mathrm{IP},0}'(t) = G(t) + F(t)$.  Then we use \eqref{e:order1} and \eqref{e:numerbound} with $c=2$ and $t = (\alpha -4/3 - 5s)\log N/\log \delta_2^{-1}$
\begin{align*}
|G(t)|  &\leq C \left( \sum_{n=2}^N (1 - \delta_n)(\lambda_1^{-1}-\lambda^{-1}_n) \delta_n^t \nu_n \right)\left(\sum_{n=2}^N  \log \delta_n^{-1}\delta_n^{t} \nu_n \right)\\
&\leq C N^{-\alpha + 10s} \left( \sum_{n=2}^N \log \delta_n^{-1} \delta_n^t \nu_n \right) \leq C N^{-2\alpha + 20s + 2/3} \leq C N^{-\alpha - 8/3 - \sigma/2},
\end{align*}
for $t \in \hat L_\alpha$. The last inequality follows because $\alpha \geq 10/3 + \sigma$ and $\sigma > 40s$.  From here it follows that for $N$ sufficiently large
\begin{align*}
-E_{\mathrm{IP},0}'(t) \geq F(t) - |G(t)| \geq C_1 N^{-11s - 2/3 - \alpha} - C_2 N^{-\alpha - 8/3} \geq C N^{-11s - 2/3 - \alpha}.
\end{align*}
\end{proof}

Our next step is to construct an approximation $T^*_{\mathrm{IP},\epsilon}$ of $T_{\mathrm{IP},\epsilon}$.  We write
\begin{align*}
E_{\mathrm{IP},0}(t) &=  \frac{\ds  \sum_{n=2}^N (1-\delta^{1/2}_n)^2(\lambda_n^{-1} + \lambda_1^{-1})\delta_n^t \nu_n }{\ds \left(\sum_{n=1}^N \delta_n^{t} \nu_n \right)\left(\sum_{n=1}^N \delta_n^{t+1} \nu_n \right) } \\
&  = \delta_2^{t} (1 - \delta_2^{1/2})^2 (\lambda_2^{-1} + \lambda_1^{-1}) \nu_2 \frac{\ds  1 +  \sum_{n=3}^N \frac{(1-\delta^{1/2}_n)^2(\lambda_n^{-1} + \lambda_1^{-1})\delta_n^t \nu_n}{(1 - \delta_2^{1/2})^2 (\lambda_2^{-1} + \lambda_1^{-1}) \delta_2^{t} \nu_2} }{\ds \left(\sum_{n=1}^N \delta_n^{t} \nu_n \right)\left(\sum_{n=1}^N \delta_n^{t+1} \nu_n \right) }.
\end{align*}
Define $T^*_{\mathrm{IP},\epsilon}$ by
\begin{align}
  N^{-\alpha} &= \delta_2^{T^*_{\mathrm{IP},\epsilon}} (1 - \delta_2^{1/2})^2 (\lambda_2^{-1} + \lambda_1^{-1}) \nu_2\notag\\
T^*_{\mathrm{IP},\epsilon} &= \frac{\alpha \log N + 2\log (1 - \delta_2^{1/2}) + \log(\lambda_2^{-1} + \lambda_1^{-1}) + \log \nu_2}{\log \delta_2^{-1}}.\label{e:t-starIP}
\end{align}

\begin{lemma}\label{l:inhatL}
Given Condition~\ref{cond:rigidity}, $T^*_{\mathrm{IP},\epsilon} \in \hat L_\alpha$.
\end{lemma}
\begin{proof}
Using Condition~\ref{cond:rigidity}
\begin{align*}
N^{-2/3-s} &\leq 1 - \frac{\lambda_1}{\lambda_2} \leq N^{-2/3+s}, \quad 1/C \leq \lambda_1^{-1} + \lambda_2^{-1} \leq C, \quad N^{-2s} \leq \nu_2 \leq N^{2s},\\
\end{align*}
we find
\begin{align*}
\frac{(\alpha -4/3 - 4s) \log N}{\log \delta_2^{-1}}  \leq T^*_{\mathrm{IP},\epsilon} \leq \frac{(\alpha -4/3 + 4s) \log N}{\log \delta_2^{-1}},
\end{align*}
for sufficiently large $N$, establishing the lemma.
\end{proof}

\begin{lemma}\label{l:est-error}
Given Conditions~\ref{cond:uppergap} and \ref{cond:rigidity},
\begin{align*}
N^{\alpha}|E_{\mathrm{IP},0}(t) - N^{-\alpha}| \leq C N^{4s -2p},
\end{align*}
for $t \in \hat L_\alpha$ and sufficiently large $N$.
\end{lemma}
\begin{proof}
  By direct calculation
  \begin{align*}
&|E_{\mathrm{IP},0}(t) - N^{-\alpha}| \leq N^{-\alpha}\left| 1- \frac{\ds  1 +  \sum_{n=3}^N \frac{(1-\delta^{1/2}_n)^2(\lambda_n^{-1} + \lambda_1^{-1})\delta_n^t \nu_n}{(1 - \delta_2^{1/2})^2 (\lambda_2^{-1} + \lambda_1^{-1}) \delta_2^{t} \nu_2} }{\ds \left(\sum_{n=1}^N \delta_n^{t} \nu_n \right)\left(\sum_{n=1}^N \delta_n^{t+1} \nu_n \right) }  \right|\\
&\leq N^{-\alpha}\frac{\ds  \sum_{n=3}^N \frac{(1-\delta^{1/2}_n)^2(\lambda_n^{-1} + \lambda_1^{-1})\delta_n^t \nu_n}{(1 - \delta_2^{1/2})^2 (\lambda_2^{-1} + \lambda_1^{-1}) \delta_2^{t} \nu_2}  + \sum_{n=2}^N \delta_n^{t} \nu_n + \sum_{n=2}^N \delta_n^{t+1} \nu_n}{\ds \left(\sum_{n=1}^N \delta_n^{t} \nu_n \right)\left(\sum_{n=1}^N \delta_n^{t+1} \nu_n \right) }\\
& + N^{-\alpha}\frac{\ds \left(\sum_{n=2}^N \delta_n^{t} \nu_n \right)\left(\sum_{n=2}^N \delta_n^{t+1} \nu_n \right) }{\ds \left(\sum_{n=1}^N \delta_n^{t} \nu_n \right)\left(\sum_{n=1}^N \delta_n^{t+1} \nu_n \right) }.
  \end{align*}
Since the denominator is at least unity, it is enough to estimate the numerators. As $\lambda_{n}^{-1} + \lambda_1^{-1} \leq 2 \lambda_1^{-1}$,
\begin{align*}
\frac{\lambda_{n}^{-1} + \lambda_1^{-1}}{\lambda_{2}^{-1} + \lambda_1^{-1}} \leq C.
\end{align*}
For $c > 0$, define $\hat I_c = \{ 3 \leq n \leq N :  \delta_n \leq \delta_2^{1+c} \}$.  We estimate
\begin{align*}
\sum_{n=3}^N \left( \frac{1-\delta_n^{1/2}}{1-\delta_2^{1/2}} \right)^2 \frac{\nu_n}{\nu_2} \left( \frac{\delta_n}{\delta_2} \right)^{t} = \left( \sum_{n \in \hat I_c} + \sum_{n \not\in \hat I_c}  \right)\left( \frac{1-\delta_n^{1/2}}{1-\delta_2^{1/2}} \right)^2 \frac{\nu_n}{\nu_2} \left( \frac{\delta_n}{\delta_2} \right)^{t}.
\end{align*}
First,
\begin{align*}
S_1 \! := \!\!\sum_{n \in \hat I_c}\left( \frac{1-\delta_n^{1/2}}{1-\delta_2^{1/2}} \right)^2 \frac{\nu_n}{\nu_2} \left( \frac{\delta_n}{\delta_2} \right)^{t} \!\!\leq\! \sum_{n \in \hat I_c}\left( \frac{1-\delta_N^{1/2}}{1-\delta_2^{1/2}} \right)^2 \frac{\nu_n}{\nu_2} \delta_2^{ct} \leq C N^{7/3+3s - c (\alpha - 4/3 - 5s)}.
\end{align*}
Here we used $\sum_n \frac{\nu_n}{\nu_2} = \beta_2^{-2} \leq N^{1+s}$ and estimated $(1 - \delta_N^{1/2})/(1-\delta_2^{1/2}) \leq C/(1-\delta_2^{1/2}) \leq CN^{2/3+s}$. If we set $c = 2$ and use the inequality $\alpha - 4/3 - 5s > 2 + \sigma - 5s$, then
\begin{align*}
S_1 \leq C N^{7/3 + 3s - 2(2 + \sigma - 5s)} \leq C N^{-5/3},
\end{align*}
as $s < \sigma/40$. Second, using that $\delta_3/\delta_2 \geq \delta_n/\delta_2$ and Condition~\ref{cond:lowergap} along with $ \delta_n > \delta_2^{1+c}$, $n \not\in \hat I_c$ ($c =2$), we consider
\begin{align*}
S_2 &: = \sum_{n \not\in \hat I_c}\left( \frac{1-\delta_n^{1/2}}{1-\delta_2^{1/2}} \right)^2 \frac{\nu_n}{\nu_2} \delta_2^{pt}, \quad \frac{1-\delta_n^{1/2}}{1-\delta_2^{1/2}}  \leq \frac{1-\delta_2^{3/2}}{1-\delta_2^{1/2}} \leq C,\\
& N - |\hat I_c| \leq N^{2s}, \quad \text{(from Lemma~\ref{l:Ic})},\\
& \delta_2^{pt} \leq N^{-(\alpha -4/3 - 5s)p} \leq N^{(-2+\sigma +5s)p} \leq N^{-2p} \text{  and},\\
& \nu_n/\nu_2 \leq N^{2s},
\end{align*}
by Condition~\ref{cond:rigidity}, and so $S_2 \leq C N^{4s-2p}$. Since $p < 1/2$, we find
\begin{align*}
S_1 + S_2 \leq C N^{4s-2p},
\end{align*}
for sufficiently large $N$.  By \eqref{e:gamma-est}, 
\begin{align*}
\sum_{n=2}^N \delta_n^t \nu_n \leq C N^{-\alpha + 4/3 + 9s} \leq C N^{-2-\sigma + 9s} \leq C N^{-2},
\end{align*}
for sufficiently large $N$ as $s < \sigma /40$. We find
\begin{align*}
N^{\alpha}|E_{\mathrm{IP},0}(t) - N^{-\alpha}| \leq C(S_1 + S_2 +  N^{-2}) \leq C N^{4s-2p}.
\end{align*}
This proves the lemma.
\end{proof}

The next lemma is a restatement of \eqref{e:IP1}
\begin{lemma}\label{l:IP1est}
Given Condition~\ref{cond:rigidity}, for $t \in \hat L_\alpha$
\begin{align*}
|E_{\mathrm{IP},1}(t)| \leq CN^{-\alpha -2/3 -\sigma/2}
\end{align*}
for sufficiently large $N$.
\end{lemma}

\begin{proposition}\label{p:Pow-conv}
Given Conditions~\ref{cond:uppergap} and \ref{cond:rigidity} for $\sigma < 1/3$ and $p < \sigma/4$, fixed, with $s < \sigma/40$
\begin{align*}
N^{-2/3}|T^*_{\mathrm{IP}, \epsilon} - T_{\mathrm{IP}, \epsilon}| \leq C N^{-2p + 15s}
\end{align*}
for $N$ sufficiently large.
\end{proposition}
\begin{proof}
We use the analog of \eqref{e:meanval} to estimate the difference  $|T^*_{\mathrm{IP}, \epsilon} - T_{\mathrm{IP}, \epsilon}| $(for some $\eta \in \hat L_\alpha$) and apply Lemmas~\ref{l:Tip}, \ref{l:E0prime}, \ref{l:inhatL}, \ref{l:est-error} and \ref{l:IP1est} to find
\begin{align*}
|T^*_{\mathrm{IP},\epsilon}&-T_{\mathrm{IP},\epsilon}| \\
&\leq  \frac{|E_{\mathrm{IP},0}(T^*_{\mathrm{IP},\epsilon}) -E_{\mathrm{IP},0}(T_{\mathrm{IP},\epsilon})|}{ |E_{\mathrm{IP},0}'(\eta)|} \leq \frac{|E_{\mathrm{IP},0}(T^*_{\mathrm{IP},\epsilon}) -N^{-\alpha} + E_{\mathrm{IP},1}(T_{\mathrm{IP},\epsilon})|}{ \max_{\eta\in \hat L_\alpha} E_{\mathrm{IP},0}'(\eta)}\\
& \leq \frac{|E_{\mathrm{IP},0}(T^*_{\mathrm{IP},\epsilon}) -N^{-\alpha}| + |E_{\mathrm{IP},1}(T_{\mathrm{IP},\epsilon})|}{ \max_{\eta\in L_\alpha} E_{\mathrm{IP},0}'(\eta)}\\
& \leq C N^{11s + \alpha + 2/3} \left( N^{-\alpha - 2p + 4s}  + N^{-\alpha - 2/3 - \sigma/2}  \right).
\end{align*}
The proposition follows.
\end{proof}

Now, we introduce probabilistic considerations as we did for the QR algorithm.

\begin{theorem}\label{t:Power-tech}
Let $H$ be an SCM and let $v$ be a random unit vector independent of $H$.  For $\alpha \geq 10/3 + \sigma$, $\sigma>0$
\begin{align*}
F_\beta^{\mathrm{gap}}(t) = \lim_{N \to \infty} \mathbb P \left( \frac{T_{\mathrm{IP},\epsilon}(H,v)}{(\alpha/2 - 2/3) \lambda_-^{1/3} d^{1/2} N^{2/3} \log N}  \leq t\right).
\end{align*}
\end{theorem}
\begin{proof}
  As was the case in the proof of Theorem~\ref{t:QR-tech}, we show the following three random variables converge to zero in probability:
  \begin{align*}
  &N^{-2/3} |T^*_{\mathrm{IP}, \epsilon} - T_{\mathrm{IP}, \epsilon}|, \quad \left|\frac{T^*_{\mathrm{IP}, \epsilon}}{N^{2/3} \log N} -  \frac{\alpha - 4/3}{ N^{2/3}\log \delta_2^{-1}}\right|, \text{ and }\\
  &\left|\frac{\alpha/2 - 2/3}{ N^{2/3} \Delta_2} -  \frac{\alpha - 4/3}{ \lambda_- N^{2/3}\log \delta_2^{-1}}\right|.
\end{align*}
We start with the first.  For $\delta > 0$
\begin{align*}
  \mathbb P\left(  N^{-2/3} |T^*_{\mathrm{IP}, \epsilon} - T_{\mathrm{IP}, \epsilon}| \geq \delta \right) &= \mathbb P\left(  N^{-2/3} |T^*_{\mathrm{IP}, \epsilon} - T_{\mathrm{IP}, \epsilon}| \geq \delta, \mathcal R_{N,s} \cap \mathcal L_{N,p} \right) \\
  &+ \mathbb P\left(  N^{-2/3} |T^*_{\mathrm{IP}, \epsilon} - T_{\mathrm{IP}, \epsilon}| \geq \delta, R^c_{N,s} \cup U^c_{N,p} \right)
\end{align*}
Provided that $s < (2/15)p$, on $\mathcal R_{N,s}$, $N^{-2/3} |T^*_{\mathrm{IP}, \epsilon} - T_{\mathrm{IP}, \epsilon}|$ tends to zero uniformly.  Then
\begin{align*}
  \limsup_{N\to\infty}\mathbb P&\left(  N^{-2/3} |T^*_{\mathrm{IP}, \epsilon} - T_{\mathrm{IP}, \epsilon}| \geq \delta \right) \\
  &\leq \limsup_{N \to \infty} \mathbb P(\mathcal R^c_{N,s}) + \limsup_{N \to \infty} \mathbb P(\mathcal U^c_{N,p}).
\end{align*}
From Theorem~\ref{t:generic}, $\limsup_{N \to \infty} \mathbb P(\mathcal R^c_{N,s}) = 0$, and letting $p \downarrow 0$, using Theorem~\ref{t:p}
we find
\begin{align*}
\limsup_{N\to\infty}\mathbb P\left(  N^{-2/3} |T^*_{\mathrm{IP}, \epsilon} - T_{\mathrm{IP}, \epsilon}| \geq \delta \right) = 0.
\end{align*}
For the second random variable:
\begin{align*}
&\left|\frac{T^*_{\mathrm{IP}, \epsilon}}{N^{2/3} \log N} -  \frac{\alpha - 4/3}{ N^{2/3}\log \delta_2^{-1}}\right| \\&=  \left| \frac{\alpha \log N + 2\log (1 - \delta_2^{1/2}) + \log(\lambda_2^{-1} + \lambda_1^{-1}) + \log \nu_2}{N^{2/3} \log N \log \delta_2^{-1}} -  \frac{\alpha - 4/3}{ N^{2/3}\log \delta_2^{-1}} \right|\\
 & = \left| \frac{2\log (1 - \delta_2^{1/2}) + \log(\lambda_2^{-1} + \lambda_1^{-1}) + \log \nu_2}{N^{2/3} \log N \log \delta_2^{-1}} +  \frac{4/3}{ N^{2/3}\log \delta_2^{-1}} \right|\\
 & = \left| \frac{ 2\log N^{2/3}(1 - \delta_2^{1/2}) + \log(\lambda_2^{-1} + \lambda_1^{-1}) + \log N \beta_2^2 - \log N \beta_1^2}{N^{2/3} \log N \log \delta_2^{-1}}  \right|:= Y_N
\end{align*}
Again, we write $\lambda_j = \lambda_- + N^{-2/3}\xi_j$, $j = 1,2$ and let $B_R$ be the event where $\|(\xi_1,\xi_2)\|_2 \leq R$.  Next let $H_{j,R}$ be event where $1/R \leq N \beta_j^2 \leq R$.  It then follows for $\delta > 0$ and sufficiently large $N$
\begin{align*}
\mathbb P ( Y_N \geq \delta ,H_{1,R} \cap H_{2,R} \cap B_R) = 0.
\end{align*}
Therefore
\begin{align}\label{e:YN}
  \begin{split}
\limsup_{N \to \infty}\mathbb P (Y_N \geq \delta) &\leq \limsup_{N \to \infty} \mathbb P(H_{1,R}^c)+ \limsup_{N \to \infty} \mathbb P(H_{2,R}^c) \\
&+ \limsup_{N \to \infty} \mathbb P(B_{R}^c).
  \end{split}
\end{align}
And because $(\xi_1,\xi_2), N \beta_1^2$ and $N \beta_2^2$ converge in distribution, if we let $R \to \infty$ in \eqref{e:YN} it follows that $\mathbb P ( Y_N \geq \delta) = 0$.  The convergence in probability of the last random variable follows directly from the proof of Theorem~\ref{t:QR-tech}.  Using Definition~\ref{def:Fbeta} we have
\begin{align*}
F_\beta^{\mathrm{gap}}(t) & = \lim_{N \to \infty} \mathbb P \left( \frac{\lambda_-^{2/3}}{d^{1/2}(\xi_2 - \xi_1)} \leq t  \right)\\
& = \lim_{N \to \infty} \mathbb P \left( \frac{T_{\mathrm{IP},\epsilon}}{(\alpha/2 - 2/3) \lambda_-^{1/3} d^{1/2} N^{2/3} \log N}  \leq t\right).
\end{align*}
\end{proof}

Finally, we establish the analogous theorem for the power method. Following Remark~\ref{r:power}, we note that $E_{\mathrm{P}}(t)$ is defined by sending $\lambda_j \to \lambda_j^{-1}$ and $H^{-1}$ satisfies the same estimates as $H$ (Theorem~\ref{t:p} and Theorem~\ref{t:generic}). We have the following theorem.
\begin{theorem}
Let $H$ be an SCM and let $v$ be a random unit vector independent of $H$. For $\alpha \geq 10/3 + \sigma$, $\sigma>0$
\begin{align*}
F_\beta^{\mathrm{gap}}(t) = \lim_{N \to \infty} \mathbb P \left( \frac{T_{\mathrm{P},\epsilon}(H,v)}{(\alpha/2 - 2/3) \lambda_+^{1/3} d^{1/2} N^{2/3} \log N}  \leq t\right).
\end{align*}
\end{theorem}

\appendix

\section{Error analysis}\label{sec:error}

In this section we establish that the halting times given above for the QR algorithm and the inverse power method are adaquate to acheive an order $\epsilon$ approximation of the smallest eigenvalue.

\subsection{QR algorithm}

The true error in the QR algorithm is
\begin{align*}
E_{\mathrm{QR}}^{\mathrm{True}}(t) := |\lambda_1 - X_{NN}(t)| = \left| \lambda_1 - \frac{ \sum_{n=1}^N \lambda_n^{-2t+1} \beta_n^2}{\sum_{n=1}^N \lambda_n^{-2t} \beta_n^2}\right| = \frac{ \sum_{n=1}^N  \Delta_2 \delta_n^t \nu_n}{\sum_{n=1}^N \delta_n^t \nu_n}.
\end{align*}
Applying Lemma~\ref{l:estimate} (see also \eqref{e:order1}) we find for $t \in L_\alpha$, given Condition~\ref{cond:rigidity}
\begin{align}\label{e:QRtrue}
   N^{-\alpha -11s + 2/3} \leq E_{\mathrm{QR}}^{\mathrm{True}}(t) \leq  N^{-\alpha + 11s + 2/3},
\end{align}
for sufficiently large $N$.  We obtain the following error estimate.

\begin{proposition}\label{p:QRtrue}
For $\alpha \geq 10/3 + \sigma$, $\epsilon = N^{-\alpha/2}$,
\begin{align*}
\epsilon^{-1} E_{\mathrm{QR}}^{\mathrm{True}}(T_{\mathrm{QR},\epsilon})
\end{align*}
converges to zero in probability, while
\begin{align*}
\epsilon^{-2} E_{\mathrm{QR}}^{\mathrm{True}}(T_{\mathrm{QR},\epsilon})
\end{align*}
converges to $\infty$ in probability.
\end{proposition}
\begin{proof}
First, given Condition~\ref{cond:rigidity}, $T_{\mathrm{QR},\epsilon} \in L_\alpha$ and for $\delta > 0$
\begin{align*}
\mathbb P \left( \epsilon^{-1} E_{\mathrm{QR}}^{\mathrm{True}}(T_{\mathrm{QR},\epsilon}) \geq \delta  \right) &= \mathbb P \left( \epsilon^{-1} E_{\mathrm{QR}}^{\mathrm{True}}(T_{\mathrm{QR},\epsilon}) \geq \delta, \mathcal R_{N,s}  \right) \\
&+ \mathbb P \left( \epsilon^{-1} E_{\mathrm{QR}}^{\mathrm{True}}(T_{\mathrm{QR},\epsilon}) \geq \delta, \mathcal R^c_{N,s}  \right).
\end{align*}
It then follows that on $\mathcal R_{N,s}$ for sufficiently large $N$ and $s < \sigma/22$
\begin{align*}
  \epsilon^{-1} E_{\mathrm{QR}}^{\mathrm{True}}(T_{\mathrm{QR},\epsilon}) \leq N^{-\alpha/2 + 11s + 2/3} \leq N^{-1}.
\end{align*}
Therefore
\begin{align*}
  \limsup_{N \to \infty} \mathbb P \left( \epsilon^{-1} E_{\mathrm{QR}}^{\mathrm{True}}(T_{\mathrm{QR},\epsilon}) \geq \delta  \right)  \leq \limsup_{N \to \infty } \mathbb P(\mathcal R^c_{N,s}) = 0.
\end{align*}
It then follows that on $\mathcal R_{N,s}$ for sufficiently large $N$ and $s < 1/33$
\begin{align*}
  \epsilon^{-2} E_{\mathrm{QR}}^{\mathrm{True}}(T_{\mathrm{QR},\epsilon}) \geq N^{1/3}.
\end{align*}
Therefore
\begin{align*}
  \limsup_{N \to \infty} \mathbb P \left( \epsilon^{-1} E_{\mathrm{QR}}^{\mathrm{True}}(T_{\mathrm{QR},\epsilon}) \geq \delta  \right)  \geq \liminf_{N \to \infty } \mathbb P(\mathcal R_{N,s}) = 1.
\end{align*}
\end{proof}

\begin{remark}\label{r:QR-true}
Define the ``true'' halting time by
\begin{align*}
T_{\mathrm{QR},\epsilon}^{\mathrm{True}}(H) := \inf \{t: E_{\mathrm{QR}}^{\mathrm{True}} \leq \epsilon \}.
\end{align*}
We omit the details, but one can show that
\begin{align*}
F_\beta^{\mathrm{gap}}(t) =  \lim_{N \to \infty} \mathbb P \left( \frac{\displaystyle T^{\mathrm{True}}_{\mathrm{QR},\epsilon}(H)}{ \displaystyle  2^{-7/6} \lambda_-^{1/3} d^{-1/2} N^{2/3}  (\log \epsilon^{-1} - 2/3) \log N}  \leq t\right).
\end{align*}
So that $ T_{\mathrm{QR},\epsilon}^{\mathrm{True}} $ has the same limiting distribution as $T_{\mathrm{QR},\epsilon}$.
\end{remark}

\subsection{Inverse power method}\label{a:IP-error}

The true error for the inverse power method is also given by
\begin{align*}
E_{\mathrm{IP}}^{\mathrm{True}}(t) &= |\lambda_1 - \lambda_{\mathrm{IP}}(t)| = \left|\lambda_1 -  \frac{\sum_{n=1}^N \lambda_n^{-2t} \beta_n^2}{\sum_{n=1}^N \lambda_n^{-2t-1} \beta_n^2}\right|.
\end{align*}
Following the calculations that led to \eqref{e:QRtrue}, given Condition~\ref{cond:rigidity}, for sufficiently large $N$ and $t \in \hat L_\alpha$,
\begin{align*}
N^{-\alpha + 2/3 - 11s} \leq E_{\mathrm{IP}}^{\mathrm{True}}(t) \leq N^{-\alpha + 2/3 + 11s}.
\end{align*}
Here we are conservative with the factor on $s$ sot hat it mirrorw \eqref{e:QRtrue}.  An analogous formula holds for the power method.  We arrive at the following propsition that is proved in the exact same way as Proposition~\ref{p:QRtrue}.
\begin{proposition}\label{p:Powertrue}
For $\alpha \geq 10/3 + \sigma$, $\epsilon = N^{-\alpha/2}$,
\begin{align*}
\epsilon^{-1} |\lambda_1 - \lambda_{\mathrm{IP}}(T_{\mathrm{IP},\epsilon})| \quad \text{ and } \quad \epsilon^{-1} |\lambda_1 - \lambda_{\mathrm{P}}(T_{\mathrm{P},\epsilon})|
\end{align*}
converge to zero in probability, while
\begin{align*}
\epsilon^{-2} |\lambda_1 - \lambda_{\mathrm{IP}}(T_{\mathrm{IP},\epsilon})| \quad \text{ and } \quad \epsilon^{-2} |\lambda_1 - \lambda_{\mathrm{P}}(T_{\mathrm{P},\epsilon})|
\end{align*}
converge to $\infty$ in probability.
\end{proposition}

\begin{remark}\label{r:IP-true}
Following, Remark~\ref{r:QR-true} define the ``true'' halting time by
\begin{align*}
T_{\mathrm{IP},\epsilon}^{\mathrm{True}} := \inf \{t: E_{\mathrm{IP}}^{\mathrm{True}}(t) \leq \epsilon \}.
\end{align*}
Again, omitting the details, one can show that
\begin{align*}
F_\beta^{\mathrm{gap}}(t) =  \lim_{N \to \infty} \mathbb P \left( \frac{\displaystyle T^{\mathrm{True}}_{\mathrm{IP},\epsilon}(H)}{ \displaystyle  2^{-7/6} \lambda_-^{1/3} d^{-1/2} N^{2/3}  (\log \epsilon^{-1} - 2/3) \log N}  \leq t\right).
\end{align*}
So that $ T_{\mathrm{IP},\epsilon}^{\mathrm{True}} $ has the same limiting distribution as $T_{\mathrm{IP},\epsilon}$.  This further justifies the definition of $ T_{\mathrm{IP},\epsilon}$.
\end{remark}

\section{Asymptotic normality of the eigenvector projections} \label{sec:proj}

This section presents the estimates that are required to prove Theorem~\ref{t:main} for the power and inverse power methods when the initial unit vector $v$ is chosen randomly.

\begin{theorem}\label{t:normalprojection}
  Let $v = v_N \in \mathbb R^N$ ($\beta = 1$) or $v = v_N \in \mathbb C^N$ ($\beta = 2$) be a unit vector\footnote{To be precise about this, fix a semi-infinite vector $w = (w_1,w_2,\ldots, w_N, \ldots)$.  Then for $y_N := (w_1,\ldots,w_N)$ define $v_N = y_N/\|y_N\|_2$.} and fix $j > 0$.  Let $u_j$ and $u_{N-j+1}$ be the eigenvectors of an SCM corresponding to $\lambda_j$ and $\lambda_{N-j+1}$.  Then for any bounded, continuous function $h: \mathbb R \to \mathbb R$
  \begin{align*}
    \mathbb E h(N^{1/2}|\langle v,u_j \rangle |) \to \mathbb Eh(|G_\beta|), \quad \mathbb E h(N^{1/2}|\langle v,u_{N-j+1} \rangle|) \to \mathbb Eh(|G_\beta|), \quad N \to \infty,
  \end{align*}
  where $G_\beta$ is either a standard normal ($\beta =1$) or a standard complex ($\beta=2$) random variable.  That is, we have convergence in distribution to $|G_\beta|$
\end{theorem}
\begin{proof}
We present the proof for $\beta =1$ and $u_j$ as the other cases are completely analogous.  From \cite[Theorem~8.2]{Bloemendal2016} it follows that
\begin{align*}
\mathbb E h(N^{1/2}|\langle v,u_j \rangle|) - \mathbb E_{\mathrm W} h(N^{1/2}|\langle v,u_j \rangle|) \to 0 ~~\mathrm{ as }~~ N \to \infty,
\end{align*}
where $\mathbb E_{\mathrm{W}}$ is the expectation with respect to the Wishart (LOE) ensemble ($X_{ij}$ are iid standard normal random variables).  And so, it is enough to prove the statement for the Wishart ensemble.  In this case it is well known that the eigenvectors are distributed with Haar measure on the orthogonal group.  Let $Y = (Y_1,\ldots,Y_N)^T$ be a vector of iid standard normal random variables.  It follows that
\begin{align*}
u_j \overset{\mathrm{dist}}{=} Y/\|Y\|_2
\end{align*}
and $\langle v,Y\rangle$ is a standard normal random variable.  And, so it suffices to show that
\begin{align}\label{e:inprob}
\left[\frac{N^{1/2}}{\|Y\|_2} -1 \right]\langle v,Y \rangle \overset{\mathrm{prob}}{\to} 0.
\end{align}
Indeed, if the difference of two random variable converges to zero in probability and the first converges in distribution then so does the second (to the same distribution).  Fix $\delta>0$ and consider for $R >0$
\begin{align*}
\mathbb P \left( \left|\frac{N^{1/2}}{\|Y\|_2} -1 \right| |\langle v,Y \rangle| \geq \delta \right)  &= \mathbb P \left( \left|\frac{N^{1/2}}{\|Y\|_2} -1 \right| |\langle v,Y \rangle| \geq \delta, |\langle v,Y \rangle| \geq R \right) \\
&+ \mathbb P \left( \left|\frac{N^{1/2}}{\|Y\|_2} -1 \right| |\langle v,Y \rangle| \geq \delta, |\langle v,Y \rangle| < R \right)\\
& \leq \mathbb P(|\langle v,Y| \geq R) + \mathbb P\left(  \left|\frac{N^{1/2}}{\|Y\|_2} -1 \right|  \geq \delta R^{-1}\right).
\end{align*}
A consequence of the Strong Law of Large Numbers (SLLN) is that the latter term tends to zero as $N \to \infty$:  The SLLN implies that $\frac{\|Y\|^2_2}{N}  \to 1$ a.s., hence $\frac{N^{1/2}}{\|Y\|_2} \to 1$ a.s. and therefore $\frac{N^{1/2}}{\|Y\|_2} \to 1$ in probability.  Then letting $R \to \infty$, \eqref{e:inprob} follows.
\end{proof}

\begin{corollary}\label{c:normalprojection}
 Theorem~\ref{t:normalprojection} holds when $v$ is a random unit vector, independent of the given SCM.  In this case $\mathbb E(\cdot)$ should be understood as the expectation with respect to both the distribution on $v$ and the SCM.
\end{corollary}
\begin{proof}
We express $\mathbb E = \mathbb E_v \mathbb E_{\mathrm{SCM}}$. Let $h$ be a bounded, continuous function $h: \mathbb R \to \mathbb R$. Then Theorem~\ref{t:normalprojection} states
\begin{align*}
\mathbb E_{\mathrm{SCM}} h(N^{1/2}|\langle v,u_j \rangle |) \to \mathbb Eh(|G_\beta|) \quad \mathrm{a.s.}.
\end{align*}
By the bounded convergence theorem
\begin{align*}
\mathbb E_v \mathbb E_{\mathrm{SCM}} h(N^{1/2}|\langle v,u_j \rangle |) \to  \mathbb E_V\mathbb Eh(|G_\beta|) = \mathbb Eh(|G_\beta|),
\end{align*}
as $N \to \infty$, and the corollary follows.
\end{proof}

\begin{proposition}\label{p:componentbound}
Given an SCM, let $v$ be a random\footnote{This also holds for deterministic $v$.} unit vector independent of the SCM.  Define $\beta_j = |\langle v, u_j \rangle|$, $j = 1,2,\ldots,N$ where $u_j$ is the $j$th eigenvector of the SCM. Fix $s> 0$ and let $\mathcal U_{N,s}$ be the set of matrices where
\begin{itemize}
\item $\beta_j \leq N^{-1/2 +s/2}$ for all $1 \leq j \leq N$, and
\item $\beta_j \geq N^{-1/2-s/2}$ for $j = 1,2,3,N-2,N-1,N$.
\end{itemize}
Then $\mathbb P(\mathcal U_{N,s}) = 1 + o(1)$ as $N \to \infty$, \emph{i.e.} these conditions hold with high probability.
\end{proposition}

\begin{proof}
The delocalization result \cite[Theorem 2.17]{Bloemendal2016} states that \emph{for deteriministic} unit vectors $v$ and all $s, D> 0$:
\begin{align*}
 \sup_v\mathbb P \left( |\langle v, u_j \rangle| > N^{-1+s} \right) \leq N^{-D}, \quad N \geq N_0(s,D).
\end{align*}
This implies
\begin{align*}
 \sup_v\mathbb P \left( |\langle v, u_j \rangle| > N^{-1+s} \text{ for no } j\right) \leq N^{-D+1}.
\end{align*}
So, let $D > 1$.  Stated another way,
\begin{align*}
\mathbb E_{\mathrm{SCM}} \mathds{1}_{\{|\langle v, u_j \rangle| > N^{-1+s}\text{ for no } j\}} \to 0, \quad N \to \infty,
\end{align*}
uniformly in $v$.  And so, taking an expectation with respect to the law of $v$ we find that
\begin{align*}
\mathbb P \left( |\langle v, u_j \rangle| \leq N^{-1+s} \text{ for all } j\right) \to 1, \quad N \to \infty
\end{align*}
Then for $j = 1,2,3,N-2,N-1,N$
\begin{align*}
\mathbb P(\beta_j \geq N^{-1/2-s/2}) = 1 + o(1), \quad N \to \infty,
\end{align*}
follows from Corollary~\ref{c:normalprojection} after applying Lemma~\ref{l:high}.

\end{proof}

\bibliographystyle{plain}
\bibliography{library}
\end{document}